\documentclass[twoside]{article}
\usepackage[a4paper]{geometry}
\usepackage[T1]{fontenc} % ou \usepackage[OT1]{fontenc}
\usepackage{RR}
\usepackage{hyperref}
  \usepackage{amsmath}
\usepackage{amssymb}
  \usepackage{paralist}
  \usepackage{graphics} %% add this and next lines if pictures should be in esp format
  \usepackage{epsfig} %For pictures: screened artwork should be set up with an 85 or 100 line screen

\usepackage{color} 
\newtheorem{theorem}{Theorem}[section]

\newtheorem{main}{Main Theorem}
\newtheorem{lemma}[theorem]{Lemma}

\newtheorem{definition}[theorem]{Definition}
\newtheorem{remark}{Remark}

\newenvironment{proof}{\paragraph{Proof:}}{\endpf\\}
\def\endpf{\hfill$\Box$\medskip}

\RRNo{8038}

\RRdate{August 1st, 2012}
\RRauthor{% les auteurs
Pierre Masci%
  \thanks[BIOCORE]{BIOCORE, INRIA Sophia
   Antipolis, BP 93, 06902 Sophia Antipolis Cedex, France. {\tt\small
           \{frederic.grognard,\,olivier.bernard\}@sophia.inria.fr} }%
  \and
Fr\'ed\'eric Grognard\thanksref{BIOCORE}
\and Eric Beno\^{\i}t \thanks{Laboratoire MIA, P\^ole Sciences et Technologie, Universit\'e de La Rochelle, Avenue Michel Cr\'epeau, 17042 La Rochelle cedex 1}\thanksref{BIOCORE}
\and Olivier Bernard\thanksref{BIOCORE} }
\authorhead{Masci, Grognard, Beno\^{\i}t \& Bernard}
\RRtitle{Competition entre phytoplankton et bacteries: exclusion et coexistence}
\RRetitle{Competition between phytoplankton and bacteria: exclusion and coexistence }
\titlehead{ Competition between phytoplankton and bacteria  }
\RRresume{La compétition pour la ressource entre micro-organismes dans 
des cultures en continu a été largement étudiée expérimentalement et théoriquement,
surtout entre bactéries modélisées par des taux de croissance de type Monod ou Contois, ou
pour le phytoplancton à travers le modèle de Droop.
Pour les populations homogènes composées de N espèces bactériennes (Monod) ou N
espèces phytoplanctoniques (Droop), avec un substrat limitant et des commandes maintenues 
constantes, les études théoriques \cite{ArmMcG1980,SmiWal1995,HsuHsu2008} ont indiqué que l'exclusion compétitive se produit: une seul
espèce remporte la compétition et élimine toutes les autres. L'espèce dont la théorie prédit la victoire est 
celle avec la concentration $s^\star$ de substrat permettant sa survie la plus basse. Ce résultat théorique a été validée
expérimentalement pour le phytoplancton \cite{TilSte1984} et les bactéries 
\cite{HanHub1980}, et a \'et\'e observé dans un lac avec des microalgues \cite{Til1977}.
Par contre, pour les espèces bactériennes décrites par un modèle Contois, la théorie  prédit la coexistence entre plusieurs espèces \cite{GroMazRap2005}.
Dans cet article nous présentons une généralisation de ces résultats en étudiant une
comp\'etition entre les trois différents types de micro-organismes: des bactéries libres
(représentées par un modèle de Monod généralisé), des bactéries fixées (représentées par un
modèle de Contois) et du phytoplancton (représenté par un modèle de Droop). Nous prouvons
que le résultat de la compétition est une coexistence entre plusieurs
espèces bactériennes fixées avec une espèce de bactéries libres ou de phytoplancton, toutes les autres espèces libres étant lessivées.
Notre démonstration est basée principalement sur l'étude de l'évolution du substrat
causée par la compétition; elle converge vers la
plus faible concentration de subsistance $s^\star $, ce qui conduit à trois types différents
des résultats de la compétition: 1. seule la meilleur bactérie libre ou le meilleur phytoplancton
exclut toutes les autres espèces; 2. seules quelques espèces bactériennes fixées
coexistent dans le chemostat, 3. une coexistence entre la meilleure espèce libre, avec
une ou plusieurs espèces fixées.
}
\RRabstract{Resource-based competition between microorganisms species in continuous
culture has been studied extensively both experimentally and theoretically,
mostly for bacteria through Monod and and Contois "constant yield" models, or
for phytoplankton through the  Droop "variable yield'' models. 
For homogeneous populations of N bacterial species (Monod) or N
phytoplanktonic species (Droop), with one limiting substrate and under
constant controls, the theoretical studies \cite{ArmMcG1980,
  SmiWal1995,HsuHsu2008} indicated that competitive exclusion occurs: only one
species wins the competition and displaces all the others. The winning species
expected from theory is the one with the lowest "substrate subsistence
concentration" $s^\star$, such that its corresponding equilibrium growth rate
is equal to the dilution rate $D$. This theoretical result was validated
experimentally with phytoplankton \cite{TilSte1984} and bacteria
\cite{HanHub1980}, and observed in a lake with microalgae \cite{Til1977}. 
On the contrary for attached bacterial species described by a Contois model, theory \cite{GroMazRap2005} predicts coexistence between several species.
In this paper we present a generalization of these results by studying a
competition between three different types of microorganisms : free bacteria
(represented by a generalized Monod mode), attached bacteria (represented by a
Contois model) and free phytoplankton (represented by a Droop model). We prove
that the outcome of the competition is a coexistence between several attached
bacterial species with a free species of bacteria or phytoplankton, all the other free species being washed out.
This demonstration is based mainly on the study of the substrate
concentration's evolution caused by competition; it converges towards the
lowest subsistence concentration $s^\star$, leading to three different types
of competition outcome: 1. only the free bacteria/ phytoplankton best
competitor excludes all other species; 2. only some attached bacterial species
coexist in the chemostat; 3. A coexistence between the best free species, with
one or several attached species.  }
\RRmotcle{competition,exclusion comp\'etitive, droop, monod, ratio-dépendence,  microorganisme, microalgues, phytoplancton}
\RRkeyword{competition, competitive exclusion, droop, variable yield model, monod, ratio-dependent, biomass-dependent, microorganism, microalgae, phytoplankton}
\RRprojets{Biocore}
\RCSophia % Sophia Antipolis M\'editerran\'ee

\newtheorem{hypothesis}{Hypothesis}
\newcommand{\limt}{\lim_{t \rightarrow +\infty}}

\begin{document}

\makeRR

%%%%%%%%%%%%%%%%%%%%%%%%%%%%%%%%%%%%%%%%
% CEP
%%%%%%%%%%%%%%%%%%%%%%%%%%%%%%%%%%%%%%%%
\section{Introduction}

\subsection{Growth of phytoplankton}

Phytoplankton is composed of microscopic plants at the basis of the aquatic
trophic chains. Phytoplankton means a broad variety of species (more than
200.000) using solar light to grow through photosynthesis. %\cite{Falkowski2007}. 
Phytoplankton plays a crucial role in nature since it is
the point from which energy and carbon enter in the food web. But it may also
be used in the future for food or biofuel production, since several
phytoplankton species turn out to have very interesting properties in terms of
protein \cite{Pulz2004,Spolaore2006} or lipid \cite{Chisti2007}
content.  
In addition to light, phytoplankton requires nutrients for its growth. 
The "paradox" of phytoplantkon species coexistence was introduced by
Hutchinson \cite{Hutchinson1961}: "The problem that is presented by the
phytoplankton is essentially how it is possible for a number of species to
coexist in a relatively isotropic or unstructered environment all competing
for the same sort of materials". In this paper we consider this question from
a theoretical viewpoint "what are the mechanisms leading to competitive
exclusion or coexistence, and to what competition outcome do they lead?". But
so far most of the competitions studies have assumed that only phytoplankton
species were engaged in the competition. However, it is clear that such
species also have to  compete with the bacteria for nutrients.  

In this paper, we study the competition between phytoplankton and
bacteria. Phytoplankton can be accurately represented by a Droop model
\cite{Dro1968, SciandraRamani,VBJM06} which accounts for their ability to
  store nutrients and to uncouple uptake and growth. Bacteria are represented
  by simpler models. They can be of two different types, described either by a
  Monod type model if they live in suspension or by a contois model if they
  are attached to a support. 

In this paper we first recall the main results available for competition of
microbial species of the same class. Then we consider the problem of 3 class
competition. After some mathematical preliminaries we state and demonstrate
our main Theorem. A discussion concludes our paper and highlights the
ecological consequences of our result.  

\subsection{The Competitive Exclusion Principle (CEP) }

\begin{center}
	"Complete competitors cannot coexist"
\end{center}
This is the formulation chosen by Hardin \cite{Har1960} to describe the
Competitive Exclusion Principle (CEP). According to him, this ambiguous
wording "is least likely to hide the fact that we still do not comprehend the
exact limits of the principle". But still, a more precise formulation is
given: if several non-interbreeding populations "do the same thing" (they
occupy the same ecological niche in Elton's sense \cite{Elt1927}) and if they
occupy the same geographic territory, then ultimately the most competitive
species will completely displace the others, which will become extinct.  

Darwin was already expressing this principle when he spoke about natural
selection (\cite{Dar1859} p.71 and 102). Scriven described and analyzed his
work in these words: "Darwin's success lay in his empirical, case by case,
demonstration that recognizable fitness was very often associated with
survival. [...] Its great commitment and its profound illumination are to be
found in its application to the lengthening past, not the distant future: in
the tasks of explanation, not in those of prediction" \cite{Scr1959}.  

Since the work of Darwin, men have tried to apprehend the limits of the
principle in different context and by different means. In the next sections we
present how mathematical models have shown their appropriateness for
predicting the outcome of competition, in the case of chemostat-controlled
microcosms.

%%%%%%%%%%%%%%%%%%%%%%%%%%%%%%%%%%%%%%%%
% CHEMOSTAT FOR CEP STUDY
%%%%%%%%%%%%%%%%%%%%%%%%%%%%%%%%%%%%%%%%
\subsection{The chemostat, a tool for studying the CEP }
\label{chem_for_CEP}

"Microbial systems are good models for understanding ecological processes at
all scales of biological organization, from genes to ecosystems"
\cite{JesBoh2005}.  
The chemostat is a device which enables to grow microorganisms under highly
controlled conditions. It consists of an open reactor crossed by a flow of
water, where nourishing nutrients are provided by the input flow, whereas both
nutrients and microorganisms are evacuated by the output flow. To keep a
constant volume in the vessel, these two flows are kept equal.  
In this paper we consider that the following conditions are imposed in the
chemostat: the medium is well mixed (homogeneous); only one substrate is
limiting for all the species, whose only (indirect) interaction is the
substrate uptake; the environmental conditions (temperature, pH, light, ...)
are kept constant, and so are the dilution rate $D$, corresponding to the
input/output flow of water, and the input substrate concentration
$s_{in}$. Figure \ref{fig:chem} represents such a chemostat. 

\begin{figure}[htp]
\begin{center}
\includegraphics[width=2in]{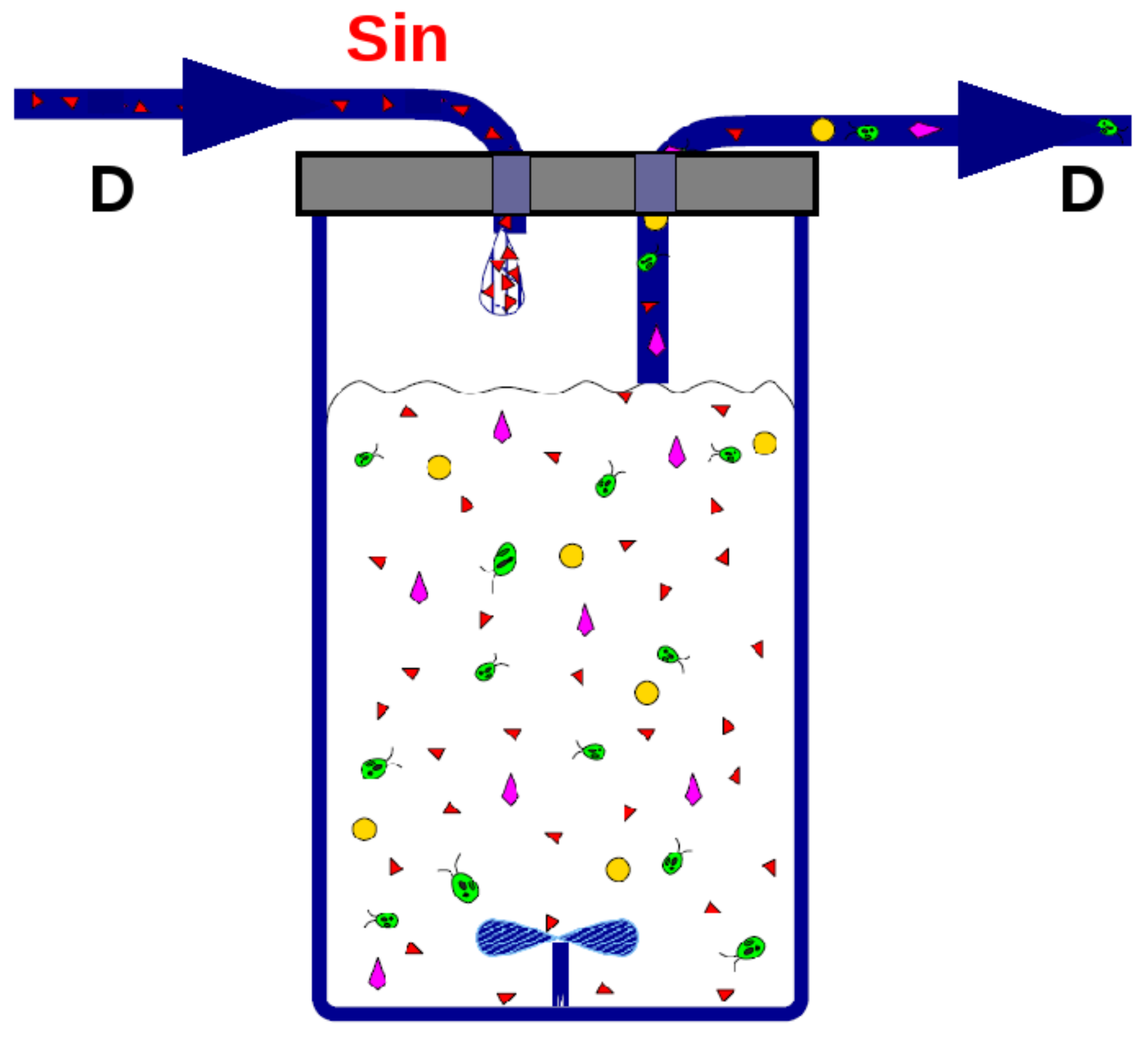}\\
\caption{A chemostat, which enables to grow microorganisms under highly
  controlled conditions. The input/output flow of water is $D$, and the input
  substrate concentration is $s_{in}$} 
\label{fig:chem}
\end{center}
\end{figure}

The chemostat has been used to study the CEP since the beginning of the
$XX^{th}$ century \cite{Gau1934}, and its experimental use has often been
coupled with mathematical models \cite{SmiWal1995}.

%%%%%%%%%%%%%%%%%%%%%%%%%%%%%%%%%%%%%%%%
% MODELS FOR GROWTH IN A CHEMOSTAT
%%%%%%%%%%%%%%%%%%%%%%%%%%%%%%%%%%%%%%%%
\subsection{Bacterial and phytoplanktonic models, and previous theoretical results on single class competition}

\subsubsection{Free bacteria growth (generalized Monod model)}

To predict the growth of bacteria  in suspension within a chemostat, Monod
developed a model \cite{Mon1942}, where the growth rates of the biomasses
$x_i$ ($i\,\in\{1,\cdots,N_x\}$ for a competition between $N_x$ species)
depend on the extracellular substrate concentration $s$.  
In the classical Monod model the growth rates $\alpha_i(s)$ are
Michaelis-Menten functions 
\[
\alpha_i(s) = \frac{s}{s+K^s_i}\alpha^m_i
\]
where $\alpha^m_i$ are the maximum growth rates in substrate replete
conditions, and $K^s_i$ are the half saturation constants. 
In this paper we consider a generalized Monod model to represent growth of
free bacteria, by using the wider class of functions verifying Hypothesis
\ref{hyp_func_x}.      
\begin{hypothesis}\label{hyp_func_x} M-model:\\
$\alpha_i(s)$ are $\mathcal{C}^1$, increasing and bounded functions such that
  $\alpha_i(0)=0$. %and \\ $\quad \sup_{s\geq 0}\alpha_i(s)=\alpha^m_i$
\end{hypothesis}
We note $\alpha^m_i$ the supremum of the growth rate:
\[
\sup_{s\geq 0}\alpha_i(s)=\alpha^m_i>0
\]

The free bacteria dynamics write 
\begin{equation}
\label{SM}
\begin{array}{l}
\dot x_i = (\alpha_i(s) - D) x_i \\
\textrm{with } s,x_i \in\,{\mathbb R}^+ \textrm{ for } i\,\in\{1,\cdots,N_x\} \textrm{ and } D\,\in {\mathbb R}^+_*.
\end{array}
\end{equation}
In this model the substrate uptake is proportional to the biomass growth for
each bacterial species, so that the total substrate uptake per time unit will
be $\sum_{i=1}^{N_x} \alpha_i(s)\frac{x_i}{a_i}$. %the substrate/biomass
                                %intracellular quota (denoted "cell quota")
                                %$1/a_i$ is kept constant for each species,
                                %and  

\subsubsection{Phytoplankton model (generalized Droop model)}

Phytoplankton is able to uncouple
substrate uptake of nutrients from the growth associated to photosynthesis
\cite{SciandraRamani}. This capacity to store nutrients can provide a
competitive advantage for the cells  that can develop in situations where
substrate and light (necessary for phytoplankton growth) are rarely available
concomitantly. This behaviour results in varying  intracellular nutrient 
quota: it is the proportion of assimilated substrate per unit of biomass
$z_k$; it can be expressed for instance in mg[substrate]/mg[biomass]. Droop
\cite{Dro1968} developed a model where these internal quotas are represented
by new 
dynamic variables $q_k$ (denoted "cell quota"). 
The substrate uptake rates $\rho_k(s)$ are assumed to depend on the
extracellular substrate while the biomass growth rates $\gamma_k(q_k)$ depend
on the corresponding cell quota.  

In the classical Droop model the functions have specified forms. Uptake rates
are Michaelis-Menten functions (\ref{rho}) of the substrate concentration: 
\begin{equation} \label{rho}
\rho_k(s) = \frac{s}{s+K^s_k} \rho^m_k 
\end{equation}
and the growth rates are Droop functions (\ref{mu}) of the cell quotas:
\begin{equation} \label{mu}
\gamma_k(q_k) =
\left\{
\begin{array}{ll}
\left( 1-\frac{Q^0_k}{q_k} \right) \bar \gamma_k \quad &\textrm{if } q_k \geq
Q^0_k \\ 
0 \quad &\textrm{if } q_k < Q^0_k
\end{array}
\right.
\end{equation}
with $\rho^m_k$ and $\bar \gamma_k$ the maximal uptake and growth rates; $K^s_k$ represent the half saturation constants, and $Q^0_k$ the minimal cell quota. 
In this paper we consider the wider class of Q-models (Quota models) verifying
Hypothesis \ref{hyp_func_z}, so that it can encompass, among others, the
classical Droop formulation \cite{Dro1968} as well as the Caperon-Meyer model
\cite{CapMey72}.  
\begin{hypothesis}\label{hyp_func_z} Q-model:
\begin{itemize}
\item $\rho_k(s)$ are $\mathcal{C}^1$, increasing and bounded functions such
  that 
  $\rho_k(0)=0$
\item $\gamma_k(q_k)$ are $\mathcal{C}^1$, increasing and bounded functions 
  for $q_k > Q^0_k>0$. When $q_k \leq Q^0_k $, $\gamma_k(q_k)=0$.
\end{itemize}
\end{hypothesis}
It directly ensues from Hypothesis \ref{hyp_func_z} that $f_k(q_k)=\gamma_k(q_k)q_k$ are increasing functions (for $q_k > Q^0_k$) which
are onto ${\mathbb R}^+_\star$, so that the inverse functions $f_k^{-1}$ are defined on ${\mathbb R}^+_\star$.

We denote $\rho^m_k$ and $\bar\gamma_k$ the supremal uptake and growth rates:
\[
\begin{array}{l}
\sup_{s\geq 0}\rho_k(s)=\rho^m_k>0\\
\sup_{q_k\geq Q^0_k}\gamma_k(q_k)=\bar\gamma_k>0
\end{array}
\]

The phytoplankton dynamics write 
\begin{equation}
\label{QM}
\begin{array}{l}
	\dot q_k = \rho_k(s) - f_k(q_k) \\
	\dot z_k = (\gamma_k(q_k) - D) z_k \\ 
\textrm{with } s,q_k,z_k \in\,{\mathbb R}^+ \textrm{ for } k\,\in\{1,\cdots,N_z\} \textrm{ and } D\,\in {\mathbb R}^+_*.

\end{array}
\end{equation}
The substrate uptake per time unit is $\sum_{k=1}^{N_z} \rho_k(s) z_k$

This model has been experimentally shown to be better suited
for phytoplankton dynamic modelling than the Monod model (\cite{VBJM06}) that
implicitly supposes that the intracellular quota is simply proportional to the
substrate concentration in the medium and which must be definitely limited to
bacterial modelling. The stability of the Q-model  
has been extensively studied in the mono-specific case (\cite{LangOyar92,OyaLan94,BerGou95}).

\subsection{Previous demonstrations of the CEP for M- and Q-models}

The advantage of Monod and Droop models is that their relative simplicity allows a mathematical analysis. 
The analyses of the M-model with $N_x$ bacterial competing species
\cite{ArmMcG1980}, and of the Droop model with 2 phytoplankton species
\cite{SmiWal1995} and then recently with $N_z$ phytoplankton  species
\cite{HsuHsu2008} led to a confirmation of the CEP in the chemostat, and to a
prediction on "who wins the competition", or "what criterion should a species
optimize to be a good competitor". In both cases, we have 
\begin{theorem} \label{th:CEP_Monod}
If environmental conditions are kept constant and the competition is not
controlled ($D$ and $s_{in}$ remain constant) in a chemostat, then the species
with lowest "substrate subsistence concentration" $s_i^{x\star}$ (or
$s_k^{z\star}$), such that its corresponding equilibrium growth rate is equal
to the dilution rate $D$, is the most competitive and displaces all the
others.  
\end{theorem}
A striking point about this result is that it permits to make predictions on
the result of a competition, only by {\it a priori} knowledge of the species
substrate subsistence concentrations $s_i^{x\star}$ (or $s_k^{z\star}$). This
latter can be determined in monospecific-culture chemostat, so that the
competition outcome can be determined before competition really occurs.  
Several experimental validations where carried out with phytoplankton
\cite{TilSte1984} and bacteria \cite{HanHub1980}. This theoretical behaviour
was also confirmed in a lake \cite{Til1977}, where the species with lowest
phosphate or silicate subsistence concentrations won the competition for
phosphate or silicate limitations.

\subsubsection{Attached bacteria model (generalized Contois model)}

In case where bacteria are not free in the medium but there is a spatial
heterogeneity ({\it e.g. they grow attached on a support, such as flocs in the
  culture medium), a ratio-dependent model is more adapted to describe
  bacterial growth.  Contois model \cite{Contois1959}
  represents such dynamics by using more complex growth functions where the
  growth rates depends on the ratio of the substrate concentration over
  biomass concentration $y_j$ ($j\,\in\{1,\cdots,N_y\}$): 
\[
\beta_j(s,y_j)= \frac{s/ y_j }{K^s_j + s/ y_j } \beta^m_j
\]
In this paper we consider the wider class of "C-model" (Contois model), wich
is more general than a ratio dependent model. It verifies the following
hypotheses: 
\begin{hypothesis}\label{hyp_func_y} C-model :\\
$\beta_j(s,y_j)$ are $\mathcal{C}^1$ functions on ${\mathbb R}^+\times
{\mathbb R}^+\setminus \{(0,0)\}$, increasing and bounded functions of $s$
(for $y_j>0$), and decreasing functions of $y_j$ (for $s_j>0$) such that 
  $\forall y_j \in {\mathbb R}^+_*, \quad \beta_j(0,y_j)=0$ and $\forall s \in {\mathbb R}^+, \displaystyle\lim_{y_j \rightarrow +\infty} \beta_j(s,y_j)=0$
\end{hypothesis}
We also need to add the following technical hyposthesis, which is verified by the classical Contois function:
\begin{hypothesis} \label{hyp:hypo_technique}
\[
\frac{\partial}{\partial y_j} \left(\beta_j(s,y_j)y_j\right)>0
\]
\end{hypothesis}
We notice that the Contois growth function is undefined in $(0,0)$, and that
Hypothesis \ref{hyp_func_y} has been built so that this property can (but does
not have to) be retained by the generalized $\beta$ function. All other
properties imposed by Hypotheses \ref{hyp_func_y} and \ref{hyp:hypo_technique}
are satisfied by the original Contois growth-rate. 

We denote $\beta^m_j(y_j)$ the supremal growth rates for biomass concentration
$y_j$: 
\[
\sup_{s\geq 0}\beta_j(s,y_j)=\beta^m_j(y_j)
\]
so that the C-species dynamics write 
\begin{equation}
\label{SB}
\begin{array}{l}
\dot y_j = (\beta_j(s,y_j) - D) y_j \\
\textrm{with } s,y_j \in\,{\mathbb R}^+ \textrm{ for } j\,\in\{1,\cdots,N_y\} \textrm{ and } D\,\in {\mathbb R}^+_*.
\end{array}
\end{equation}
In this model, like in the M-model, the substrate/biomass intracellular quotas $b_j$ are supposed to be constant for each species, so that the substrate uptake rates are proportional to the growth rates with a factor $1/b_j$.

\subsubsection{Coexistence result for competition between C-species}

Competition between several C-species was studied \cite{GroMazRap2005} and led
to a coexistence at equilibrium with the substrate at a level $s^{y\star}$
depending on the input substrate concentration $s_{in}$ and the dilution rate
$D$. The species share the available substrate.  
%\[
%s^{y\star} + \sum_{j=1}^{n_y} y_j^\star = s_{in}
%\]
%where all the equilibrium values $s^{y\star}$ and $y_j^\star$ depend on the input substrate concentration $s_{in}$. 
To be more precise we must define the "$s_0$-compliance" concept:
$s_0$-compliant species are the species able to have a growth rate equal to
the dilution rate $D$ with a substrate concentration $s_0$. The results of
\cite{GroMazRap2005} show that all the "$s^{y\star}$-compliant" species
coexist in the reactor at equilibirum, and all the others are washed out, as
they cannot grow fast enough with substrate concentration $s^{y\star}$.  

\begin{definition}\label{def:compliance}
A species $x_i, y_j$ or $z_k$ is $s_0$-compliant if it is able to reach a growth rate equal to the dilution rate $D$ with a substrate concentration $s_0$. 
\end{definition}

\subsubsection{Competition and coexistence - towards a new paradigm}

Following these results an interrogation arises: 
\begin{center}
	$\ll$ What would be the result of a competition between "competitive"
        free bacterial and microalgal species, and "coexistive"  attached
        bacterial species? Competitive exclusion? Coexistence? $\gg$ 
\end{center}
The aim of this paper is to provide an answer to this question, and to give
insight into the mechanisms forcing the outcome of such a competition. This
answer leads to a broader view and understanding of competitive exclusion and
coexistence mechanisms, following the words of Hardin \cite{Har1960}: "To
assert the truth of the competitive exclusion principle is not to say that
nature is and always must be, everywhere,"red in tooth and claw." Rather, it
is to point out that {\it every} instance of apparent coexistence must be
accounted for. Out of the study of all such instances will come a fuller
knowledge of the many prosthetic devices of coexistence, each with its own
costs and its own benefits."

\subsection{A generalized model for competition between several phytoplankton
  and bacteria species growing according to different kinetic models} 
\label{sect_pres-droop}

The generalized model for competition between all bacteria and phytoplankton
species is an aggregation of these models, which alltogether give the
following substrate dynamics, subject to substrate input, output, and uptake
rates: 
\begin{equation}
\label{s_model}
\dot s = D (s_{in}-s) - \sum_{i=1}^{N_x} \alpha_i(s) \frac{x_i}{a_i} -
\sum_{j=1}^{N_y} \beta_j(s,y_j) \frac{y_j}{b_j} - \sum_{k=1}^{N_z} \rho_k(s)
z_k 
\end{equation}
The parameters related to the nutrient flow are the dilution rate $D>0$ and
the input substrate concentration $s_{in}>0$, which are both assumed to be
constant.  

To simplify notations we can remark that this system can be normalized with
$a_i=b_j=1$, when considering the change of variables $\tilde x_i =
\frac{x_i}{a_i}$ and $\tilde y_j = \frac{y_j}{b_j}$ (note that all the
hypotheses are still satisfied). 
We obtain system (\ref{model_norm}) where variables $x_i$ and $y_j$ are now
expressed in substrate units.  
\begin{equation}
\label{model_norm}
\begin{array}{l}
\left\{
\begin{split}
	\dot s &= D (s_{in}-s) - \sum_{i=1}^{N_x} \alpha_i(s) x_i - \sum_{j=1}^{N_y} \beta_j(s,y_j) y_j - \sum_{k=1}^{N_z} \rho_k(s) z_k \\
	\dot x_i &= (\alpha_i(s) - D) x_i \\
	\dot y_j &= (\beta_j(s,y_j) - D) y_j \\
	\dot z_k &= (\gamma_k(q_k) - D) z_k \\
	\dot q_k &= \rho_k(s) - f_k(q_k) 
\end{split}
\right. 
\\
\textrm{with } \quad f_k(q_k) = \gamma_k(q_k) q_k \\
\textrm{and } s,q_k \in \,{\mathbb R}^+, \quad s_{in},D\,\in {\mathbb R}^+_*.
\textrm{and } x_i(0),y_j(0),z_k(0) \in\,{\mathbb R}^+_* \textrm{ for } i \in \{1,\cdots,N_x\}, \\
j \in \{1,\cdots,N_y\},k \in \{1,\cdots,N_z\} \\

\end{array}
\end{equation}
Note that the results obtained in this paper apply also on the simple M- only,
Q-only, and C-only competition models, or on a model with two of these three
kind of species. 

%In the rest of this paper we will consider positive initial biomasses:
%\begin{hypothesis} \label{hyp:positivity}
%$
%\begin{array}{l}
%\forall i \in \{1,\cdots,N_x\}, x_i(0)>0 \\
%\forall j \in \{1,\cdots,N_y\}, y_j(0)>0 \\
%\forall k \in \{1,\cdots,N_z\}, z_k(0)>0
%\end{array}
%$
%\end{hypothesis}
%Indeed, having some of these biomasses equal to zero at the initial time is equivalent to having less species in the model.

\subsection{Other coexistence mechanisms, and competition control}

This introduction wouldn't be complete without a short review of what has been done concerning other coexistive models, or the control of competition.

Following the question arised by Hutchinson \cite{Hutchinson1961} concerning
the "paradox of the phytoplankton", a large amount of work has been done to
explore the mechanisms that enable 
coexistence, mainly for models derived from the Monod model. It has been shown to occur in multi-resource models \cite{LeoTum75,HsuCheHub81}, in
case of non instantaneous growth \cite{FSW89}, in some turbidity operating conditions \cite{DelSmi03}, a
crowding effect \cite{DelAngSon03}, or variable yield \cite{APW03} (not in the
Droop sense). \cite{Wil90} and \cite{FreSte1981} also presented several
mechanisms which can mitigate the competition between microorganisms and
promote coexistence.  

In other papers (\cite{GouRob2005}, \cite{LeeLiSmi2003} and
\cite{RaoRox1990}), controls were proposed to "struggle against the struggle
for existence" (that is, to enable the coexistence of complete
competitors). These controls indicate how to vary the environmental conditions
in order to prevent the CEP from holding : some time varying or
state-depending environmental conditions can enable
coexistence. \cite{MasGroBer08} propose a theoretical way of driving
competition, that is, of choosing environmental conditions for which the
competitiveness criterion changes.

\section{Mathematical preliminaries}

%%%%%%%%%%%%%%%%%%%%%%%%%%%%%%%%%%%%%%%%
% BOUNDED MODEL
%%%%%%%%%%%%%%%%%%%%%%%%%%%%%%%%%%%%%%%%

\subsection{The variables are all bounded}
\label{sec:bounds}

Throughout this paper we study the evolution of one solution of system
(\ref{model_norm}) with initial condition $(s(0), x_1(0), \hdots,
x_{N_x}(0),$\\ 
$y_1(0),\hdots,y_{N_y}(0),q_1(0), \hdots, q_{N_z}(0), z_1(0), \hdots,
z_{N_z}(0))$ where $x_i(0) > 0, y_j(0)>0,$ $z_k(0)>0$. 
In this section we study the boundedness of the variables. 
First, the variables all stay in ${\mathbb R}^+$, as their dynamics are non
negative when the variable is null.  

Then we know that the biomasses remain positive:
\begin{lemma}
\[
\begin{array}{l}
\forall i \in \{1,\cdots,N_x\}, x_i(0)>0 \Leftrightarrow \forall t, x_i(t)>0\\
\forall j \in \{1,\cdots,N_y\}, y_j(0)>0 \Leftrightarrow \forall t, y_j(t)>0\\
\forall k \in \{1,\cdots,N_z\}, z_k(0)>0 \Leftrightarrow \forall t, z_k(t)>0
\end{array}
\]
\end{lemma} \label{lem:positivity}
\begin{proof}
Because of the lower bounds on the dynamics ($\dot x_i > -Dx_i$ for the free
bacteria for example), the biomasses are lower bounded by exponentials
decreasing at a rate $D$: 
\[
\forall t, x_i(t)>x_i(0) e^{-Dt}>0
\]
\end{proof}

Then, to upperbound the variables we define
\[
	M = s + \sum_{i=1}^{N_x} x_i + \sum_{j=1}^{N_y} y_j + \sum_{k=1}^{N_z} q_k z_k
\]
the total concentration of intra and extracellular substrate in the chemostat.
The computation of its dynamics gives
\begin{equation} \label{eq:dot-M}
\dot M = D(s_{in}-M)
\end{equation}
so that $M$ converges exponentially towards $s_{in}$. This linear convergence implies the upper boundedness of $M$:
\[
\forall t \geq 0, \quad M(t) \leq M^m = \max(M(0),s_{in})
\]
Then $s$, $x_i$,$y_j$ and $q_k z_k$ are also upper bounded:
\begin{equation} \label{eq:s-low-Mm}
\begin{array}{l}
\forall t \geq 0, \quad s(t) \leq M(t) \leq M^m \\
\forall i \in \{1,\cdots,N_x\}, \forall t \geq 0, \quad x_i(t) \leq M(t) \leq M^m \\
\forall j \in \{1,\cdots,N_y\}, \forall t \geq 0, \quad y_j(t) \leq M(t) \leq M^m  \\
\forall k \in \{1,\cdots,N_z\}, \forall t \geq 0, \quad q_k (t) z_k(t) \leq M(t) \leq M^m
\end{array}
\end{equation}
We are now interested in the boundedness of the Q-model's cell quotas $q_k$ and biomasses $z_k$
\begin{lemma}
$\forall k, $ the $q_k$ variables are upper bounded by $\max(f_k^{-1}(\rho_k(M^m)), q_k(0))$
\end{lemma}
\begin{proof}
For any $q_k > f_k^{-1}(\rho_k(M^m))$ there is an upper bound on $\dot q_k$:
\[
\dot q_k = \rho_k(s) - f_k(q_k) \leq \rho_k(s) - \rho_k(M^m) \leq 0
\]
so that $s\leq M^m$ implies that $q_k$ cannot increase if it is higher than $f_k^{-1}(\rho_k(M^m))$.
\end{proof}

\begin{lemma}
$\forall k, $ the $z_k$ variables are upper bounded by 
\begin{equation} \label{xibound}
z^m_k = \max\left(\frac{M^m}{\gamma_k^{-1}(D)}, z_k(0)\right)
\end{equation}
with the convention that $\gamma_k^{-1}(D)=+\infty$ if $\bar\gamma_k\leq D$
\end{lemma}
\begin{proof}
As $q_k z_k$ is upper bounded by $M^m$, there is an upper bound on $\dot z_k$ 
\[
\dot z_k=\left(\gamma_k(q_k)-D\right)z_k \leq \left(\gamma_k\left(\frac{M^m}{z_k}\right)-D\right)z_k
\]
so that $z_k$ cannot increase if it is larger than $\frac{M^m}{\gamma_k^{-1}(D)}$. 
% This last bound is only valid if $\bar\gamma_k>D$ (otherwise $x_i$ monotonically decreases towards $0$). 
\end{proof}

\begin{lemma} \label{lem:hat-s}
After a finite time $t_0$ there exists a lower bound $\hat s>0$ for $s$.
\end{lemma}
\begin{proof}
% If $s(0) = 0$ we have $\dot s = D s_{in}$, so that a finite time $t_0$ exists such that $s(t_0)>0$.
% Now,

With hypothesis \ref{hyp:hypo_technique}, and as the biomasses are upper bounded, we see that $\dot s$ can be lower bounded
\[
\dot s \geq D(s_{in}-s) -\sum_{i=1}^{N_x} \alpha_i(s) x_i^m -\sum_{j=1}^{N_y} \beta_j(s,y^m_j) y^m_j - \sum_{k=1}^{N_z} \rho_k(s) z^m_k = \phi(s)
\]
where $\phi$ is a decreasing function of $s$, with $\phi(0) = D s_{in}$ and $\phi(s_{in})<0$.
By continuity of the $\phi$ function, there exists a positive value $\hat
s<s_{in}$ such that $\phi(\hat s) = D s_{in} / 2$. The region where $s \geq
\hat s$ is therefore positively invariant. Also $s$ is increasing for any
value lower than $\hat s$ with $\dot s \geq D s_{in}/2$ so that $s(t)$ reaches
$\hat s$ after some finite time $t_0$. 
% and finally, after some finite time, we know that $s$ is lower bounded by $\hat s = \min\{\tilde s;s(0)\}$, which is higher than zero if $s(0)>0$.
% 
% In the particular case where $s(0) = 0$ we see that $\dot s(0) = D s_{in}$,
% so that there exists a time $t_0$ such that $s(t_0)>0$. Then we can choose
% $\hat s = \min\{\tilde s;s(t_0)\} > 0$. 
\end{proof}

\begin{remark} This lemma eliminates any problem that could have arisen from
  the problem of definition of $\beta_j(s,y_j)$ in $(0,0)$. After the finite
  time $t_0$, no solution can approach this critical value anymore. 
\end{remark}

\begin{lemma}\label{qbound} \label{lem:qibound}
There exists a finite time $t_1\geq 0$ such that for any time $t\geq t_1$, $q_k(t) \in (Q^0_k,Q^m_k)$ 
% along the solutions of (\ref{droop}), 
with $Q^m_k = f_k^{-1}(\rho^m_k)$.
% which is the theoretical maximal cell quota.
\end{lemma}
\begin{proof}
If $q_k(t) \geq Q^m_k$, then we have 
\[
\dot q_k \leq \rho_k(s) - f_k(Q^m_k) \leq \rho_k(M^m) - \rho^m_k < 0
\] 
for all $q_k \in [Q^m_k,q_k(0)]$, so that $q_k(t) < Q^m_k$ in finite time $t_1$ and for any $t\geq t_1$.

If $q_k(t) \leq Q^0_k$ with $t>t_0$ (defined in Lemma \ref{lem:hat-s}), then we have that  
\[
\dot q_k = \rho_k(s) \geq \rho_k(\hat s) > 0
\]
for all $q_k\,\in\,[q_k(t_0),Q^0_k]$, so that $q_k(t) > Q_k^0$ in finite time $t_1$ and for any $t\geq t_1$.

% The invariance of the interval is then checked by noticing that $\dot q_k \geq 0$ (resp. $<0$) when $q_k=Q_k^0$ (resp. $q_k=Q_k^m$). 
\end{proof}

This lemma is biologically relevant since minimum and maximum cell quotas are indeed known characteritics of phytoplankton species. 
For the rest of this paper we will consider that all the $q_k$ are in the $(Q^0_k,Q^m_k)$ intervals.

\begin{remark}
In the classical case of Michaelis-Menten uptake rates (\ref{rho}) and Droop growth rates (\ref{mu}) we have: 
\[
Q^m_k = Q^0_k + \frac{\rho^m_k}{\bar \gamma_k}
\]
\end{remark}

%%%%%%%%%%%%%%%%%%%%%%%%%%%%%%%%%%%%%%%%
% DEFINITION OF Qi and Si
%%%%%%%%%%%%%%%%%%%%%%%%%%%%%%%%%%%%%%%%
\subsection{From a "substrate" point of view... (How substrate concentration influences the system)}
\label{sec_Qi-Si}

Since model (\ref{model_norm}) is of dimension $1+N_x+N_y+2 N_z$, it is hard
to handle directly. In this section we introduce functions which clarify how
the $q_k$ and $y_j$ dynamics are influenced by $s$. This will enable us to
focus on the substrate concentration evolution, and thus reduce the dimension
in which the system needs to be analyzed.

\subsubsection{Internal cell quotas $q_k$ are driven by the substrate concentration $s$}
%\label{sec_Qi-Si}

% The mapping from $q_k$ on the $s$ space and its inverse are defined by the two key functions $Q_k(s)$ and $S^z_k(q_k)$ that we  now introduce. 
It is convenient to introduce the functions 
\begin{equation} \label{Qi}
% Q_k(s) = f_k^{-1} \circ \rho_k
Q_k(s) = f_k^{-1} (\rho_k(s))
\end{equation}
and
\begin{equation}
S^z_k(q_k) = Q_k^{-1}(q_k)
\end{equation}
With Hypothesis \ref{hyp_func_z} it is easy to check that $Q_k$ is defined,
continuous, increasing from $(0,+\infty)$ to $(Q_k^0,Q_k^m)$, so that $S^z_k$
is also well defined, continuous and increasing from $(Q_k^0,Q_k^m)$ to
$(0,+\infty)$. The $\dot q_k$ equation can then be written 
\begin{equation} \label{dot-qi_Qi}
\dot q_k = f_k(Q_k(s)) - f_k(q_k)
\end{equation}
or
\begin{equation}
\dot q_k = \rho_k(s) - \rho_k(S^z_k(q_k))
\end{equation}
Since $f_k(q_k)$ and $\rho_k(s)$ are increasing functions, we see how the dymanics of $q_k$ is influenced by the sign of $Q_k(s)-q_k$ (or $s-S^z_k(q_k)$):
\begin{equation} \label{eq:sign}
sign(\dot q_k) = sign(Q_k(s)-q_k) = sign(s-S^z_k(q_k))
\end{equation}
For a given constant substrate concentration $s$, the equilibrium value of $q_k$ is $Q_k(s)$. Conversely, $s$ must be equal to $S^z_k(q_k)$ for $q_k$ to be at equilibirum.

Function $Q_k$ realizes a mapping from the substrate axis to the cell quota
axis. Functions $S^z_k$ realizes a mapping from the cell quota axis to the
substrate axis. An illustration of the cell quotas behaviour is presented in
Figure \ref{fig_Qi-Si}.

% To conclude this section, let us explain the usefulness of these
% transformations. Instead of having separate information on $N+1$ axes ($s$,
% and the $N$ cell quotas $q_k$) we obtain a one dimensional view on these
% $N+1$ dynamics, where all the $S^z_k(q_k)$ are attracted by $s$. We have
% turned a complex $2N+1$ dimensional problem into a more comprehensible one:
% "how do $s$ and the $S^z_k(q_k)$ behave on the substrate axis, and what are
% the consequences for the $x_i$ biomasses?" \\ 

\begin{figure}[htp]
\begin{center}
\includegraphics[width=2.5in]{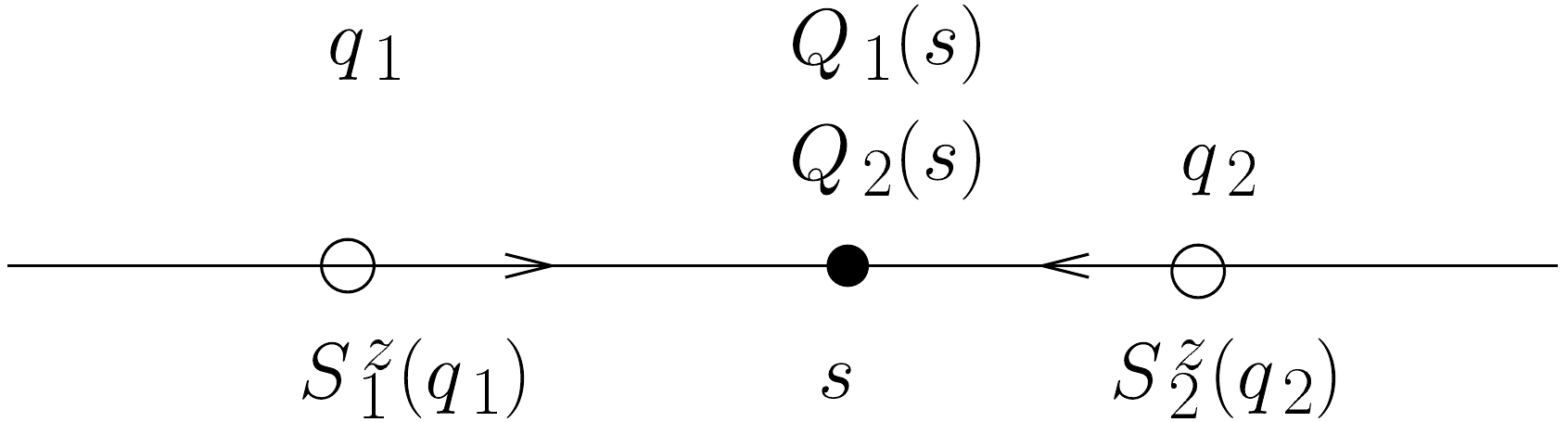}\\
\caption{Two equivalent statements: "$q_k$ goes towards $Q_k(s)$" and
  "$S^z_k(q_k)$ goes towards $s$" (see the sign Property (\ref{eq:sign})). The
  latter permits a one dimensional view of the $s$ and $q_k$ dynamics, on the
  substrate axis.} 
\label{fig_Qi-Si}
\end{center}
\end{figure}

%%%%%%%%%%%%%%%%%%%%%%%%%%%%%%%%%%%%%%%%
% DEFINITION OF Yj and Sj
%%%%%%%%%%%%%%%%%%%%%%%%%%%%%%%%%%%%%%%%
\subsubsection{How the biomasses $y_j$ are driven by the substrate concentration $s$}
\label{sec_Yj_Sj}

% The mapping from $q_k$ on the $s$ space and its inverse are defined by the two key functions $Q_k(s)$ and $S^z_k(q_k)$ that we  now introduce. 
For the C-species, it is also convenient to introduce functions $Y_j(s)$:
\begin{equation} \label{Yj}
\begin{array}{l}
\textrm{if } \beta_j(s,0)>D, \textrm{ then $Y_j(s)$ is defined by } \beta_j(s,Y_j(s)) = D\\
\textrm{if } \beta_j(s,0) \leq D, \textrm{ then } Y_j(s)=0
\end{array}
\end{equation}
and the inverse $S^y_j(y_j)$ functions:
\begin{equation}
\begin{array}{l}
\forall y_j>0, 
\left\{
\begin{array}{l}
\textrm{if } \exists  s_0 \textrm{ s.t. } \beta_j(s_0,y_j)>D, \textrm{ then $S^y_j(y_j)$ is defined by } \beta(S^y_j(y_j),y_j) = D\\
\textrm{else, } S^y_j(y_j)=+\infty
\end{array}
\right. \\
S^y_j(0) = \inf_{y_j>0} S^y_j(y_j)
\end{array}
\end{equation}
The values of $s$ such that $Y_j(s)=0$ correspond to values where the
substrate is too low for $y_j$ to survive ($y_j$ is not $s$-compliant at these
values). The values of $y_j$ such that $S^y_j(y_j)=+\infty$ correspond to
levels of biomass $y_j$ that cannot be sustained independently of the
substrate level.  

With Hypothesis \ref{hyp_func_y} it is easy to check that $Y_j$ is defined,
continuous, increasing from $\left(S^y_j(0),+\infty \right)$ to $\displaystyle
\left(0,\sup_{s\geq0} Y_j(s) \right)$, so that $S^y_j$ is also well defined,
continuous and increasing from $\displaystyle \left( 0,\sup_{s\geq0} Y_j(s)
\right)$ to $\left(S^y_j(0),+\infty \right)$.  

The $\dot y_j$ equation can then be written
\begin{equation} \label{dot-yj_Yj}
\dot y_j = (\beta_j(s,y_j) - \beta_j(s,Y_j(s)))  y_j
\end{equation}
or
\begin{equation}
\dot y_j = (\beta_j(s,y_j) - \beta_j(S^y_j(y_j),y_j)) y_j
\end{equation}
Thus with $y_j$ positivity (see Lemma \ref{lem:positivity}) we see how the
dymanics of $y_j$ are influenced by the sign of $Y_j(s)-y_j$ (or
$s-S^y_j(y_j)$): 
\begin{equation} \label{eq:sign_y}
sign(\dot y_j) = sign(Y_j(s)-y_j) = sign(s-S^y_j(y_j))
\end{equation}
For a given constant substrate concentration $s$, the equilibrium value of
$y_j$ is $Y_j(s)$. Conversely, $s$ must be equal to $S^y_j(y_j)$ for $y_j$ to
be at equilibirum. 

Function $Y_j$ realizes a mapping from the substrate axis to the cell quota
axis. Functions $S^y_j$ realizes a mapping from the cell quota axis to the
substrate axis. An illustration of the biomasses behaviour is presented in
Figure \ref{fig_Yj-Sj}.  

\begin{figure}[htp]
\begin{center}
\includegraphics[width=2.5in]{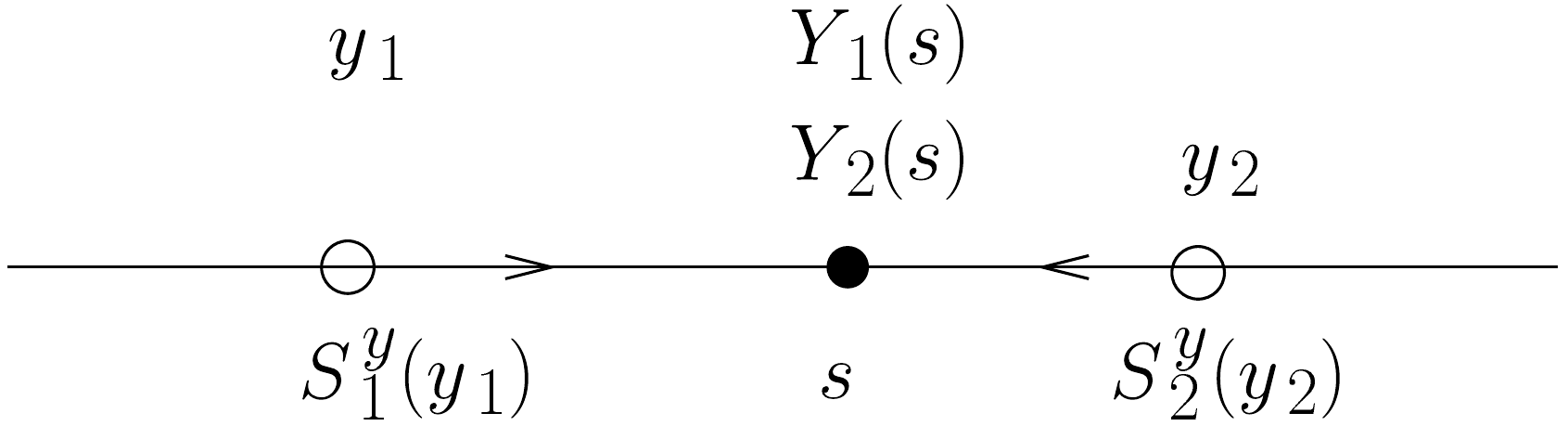}\\
\caption{Two other equivalent statements: "$y_j$ goes towards $Y_j(s)$" and
  "$S^y_j(y_j)$ goes towards $s$" (see the sign Property
  (\ref{eq:sign_y})). The latter permits a one dimensional view of the $s$ and
  $y_j$ dynamics, on the substrate axis.} 
\label{fig_Yj-Sj}
\end{center}
\end{figure}
Finally, with Figures \ref{fig_Qi-Si} and \ref{fig_Yj-Sj} we obtain a one
dimensional view of the $s$, $q_k$ and $y_j$ dynamics on the substrate
axis. The demonstration presented in this paper ensues mainly from this one
dimensional view of the system.

%%%%%%%%%%%%%%%%%%%%%%%%%%%%%%%%%%%%%%%%
% LEMMA CONVERGENCE s,q
%%%%%%%%%%%%%%%%%%%%%%%%%%%%%%%%%%%%%%%%
\subsection{The convergence of $s$ is related to the convergence of $q_k$ and $y_j$}

\begin{lemma} \label{conv_s-qi}
	In system (\ref{model_norm}) the five following properties are equivalent for any $s_0>\min_j(S^y_j(0))$:
	\begin{compactenum}[i)] 	
	\item
	$\lim_{t \rightarrow +\infty} s(t) = s_0$
	\item
	$\forall i, \lim_{t \rightarrow +\infty} q_k(t) = Q_k(s_0)$
	\item
	$\exists i, \lim_{t \rightarrow +\infty} q_k(t) = Q_k(s_0)$
	\item
	$\forall j, \lim_{t \rightarrow +\infty} y_j(t) = Y_j(s_0)$
	\item
	$\exists j, \lim_{t \rightarrow +\infty} y_j(t) = Y_j(s_0)>0$
	\end{compactenum}
	When $\lim_{t \rightarrow +\infty} s(t) = s_0 \leq \min_j(S^y_j(0))$, all the $q_k(t)$ converge to $Q_k(s_0)$ and the $y_j(t)$ to $Y_j(s_0)=0$.
\end{lemma}

\begin{proof}
	In the case $s_0>\min_j(S^y_j(0))$ we successively demonstrate five implications. \\
	$i => ii$ and $i => iv$: straightforward with the attraction
        (\ref{dot-qi_Qi}) of $q_k$ by $Q_k(s)$, and the attraction
        (\ref{dot-yj_Yj}) of $y_j$ by $Y_j(s)$. Note that $y_j(0)$ cannot be
        null (Lemma \ref{lem:positivity}).\\ 
	$ii => iii$ and $iv => v$: trivial implications. \\
	$iii => i$ (and $v=>i$): we equivalently demonstrate that the simultaneous convergence of $q_k$ (resp. $y_j$) and non convergence of $s$ lead to a contradiction.
% demonstrate the equivalent implication that the non convergence of $s$ causes the non convergence of $q_k$.

	\begin{figure}[htp]
	\begin{center}
	\includegraphics[width=4in]{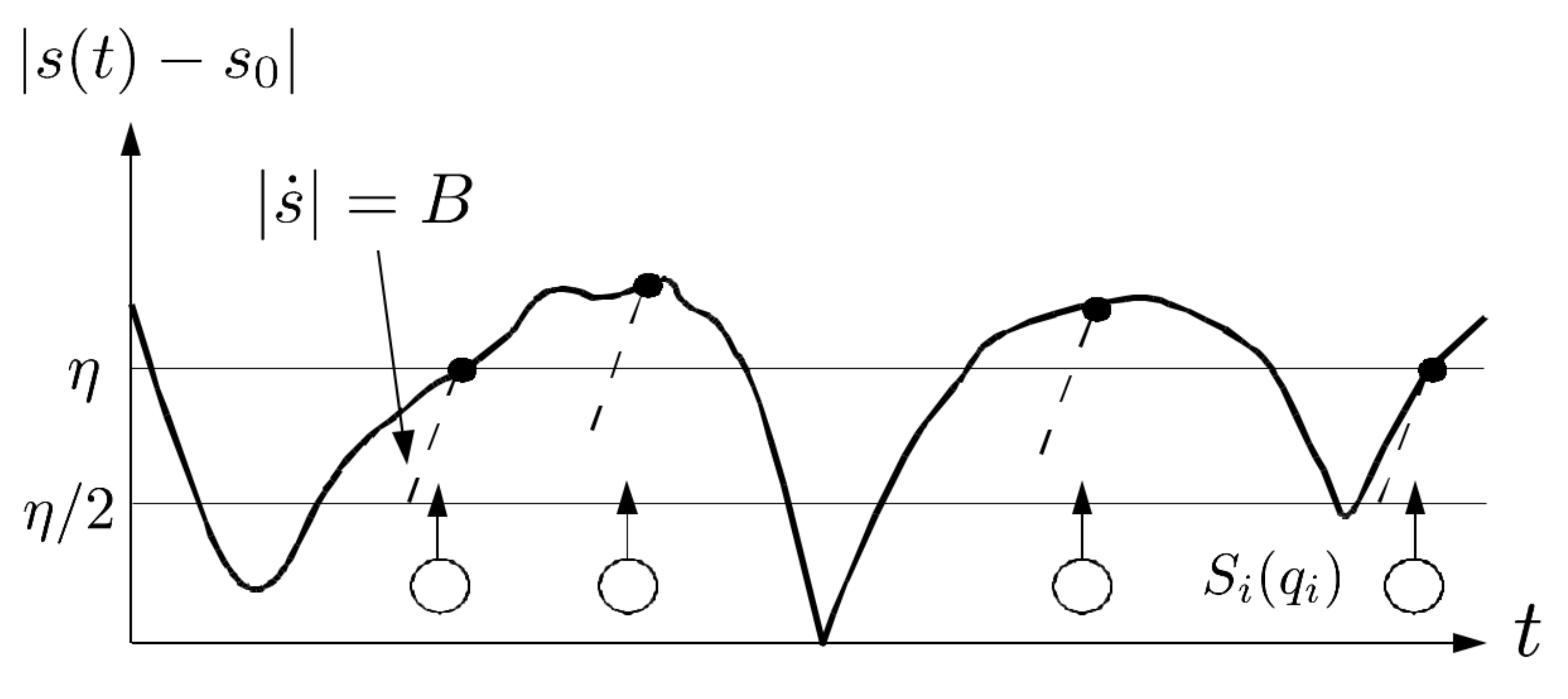}\\
	\caption{Visual explanation of the demonstration of Lemma \ref{conv_s-qi}.
	$s$ is repeatedly escaping a $\eta$-interval around $s_0$ (\emph{$\bullet$}). 
	Because $|\dot s|$ is upper bounded by $B$,  then $s$ is out of the $\eta/2$-interval during non negligible time intervals (\emph{dashed lines} represent $|\dot s| = B$). 
	$q_k$ (resp. $y_j$) is repeatedly attracted away from $Q_k(s_0)$ (resp. $Y_j(s_0)$) by $Q_k(s)$ (resp. $Y_j(s)$) (\emph{arrows})
% 	(attraction represented by \emph{arrows})
	}
	\label{fig_conv-s-qi}
	\end{center}
	\end{figure}

%%%%%
% i) s does not converge towards s_0
%%%%%

% 	\subsubsection*{Substep 1.1: $s$ is repeatedly escaping a $\eta$-interval around $s_0$.\\} 
	If $s$ does not converge towards $s_0$, it is repeatedly out of a $[s_0-\eta, s_0+\eta]$ interval, denoted $\eta$-interval:
	\[
	\exists \eta>0, \forall t>0, \exists t^s>t, |s(t^s)-s_0|>\eta
	\]
	In Figure \ref{fig_conv-s-qi}, $t^s$ time instants are represented by \emph{$\bullet$}.

%%%%%
% ii) s 'away' repeatedly
%%%%%
% 	\item 
% 	\subsubsection*{Substep 1.2: $s$ is out of the $\eta/2$-interval during non negligible time intervals.\\}

	We can then use the upper-bounds (\ref{eq:s-low-Mm}) and (\ref{xibound}) on $s$, $x_i$, $y_j$ and $z_k$ to show the boundedness of the $s$ dynamics
% boundedness of the $s$ dynamics in the case of mass balance equilibrium (\ref{MBE_eq}), and the bacause $s \leq s_{in}$ :
	\begin{equation} \label{dot-s_bound}
	D(s_{in}-M^m)-\sum_{i=1}^{N_x} \alpha^m_i x^m_i - \sum_{j=1}^{N_y} \beta^m_j(0) y^m_j - \sum_{k=1}^{N_z} \rho^m_k z^m_k \leq \dot s \leq D s_{in}
	\end{equation}
	so that 
	\[
	|\dot s| \leq B
	\]
	with $\displaystyle B = \max\left(Ds_{in}, - D(s_{in}-M^m)+\sum_{i=1}^{N_x} \alpha^m_i x^m_i + \sum_{j=1}^{N_y} \beta^m_j(0) y^m_j +\sum_{k=1}^{N_z} \rho^m_k z^m_k \right) $. \\
	Then, every time $s$ is out of the $\eta$-interval, it must also have
        been out of the $\eta/2$-interval during a time interval of minimal
        duration $A(\eta)=\frac{1}{B}\frac{\eta}{2}$. (For a visual
        explanation see the \emph{dashed lines} of Figure \ref{fig_conv-s-qi},
        representing the increase caused by $|\dot s| = B$).

%%%%%
% iii) q_k cannot converge
%%%%%
% 	\item
% 	\subsubsection*{Substep 1.3: $q_k$ is repeatedly attracted away from $Q_k(s_0)$ by $Q_k(s)$, so that their is a contradiction with the convergence towards $Q_k(s_0)$ hypothesis.\\}

If for some $t^s$ we have $s(t^s) \geq s_0+\eta$, we then have that
$s(t^s)\geq s_0+\eta/2$ during the whole time-interval $[t^s-A(\eta),t^s]$. 
% We study the case where $s(t^s) \geq s_0+\eta/2$, and the rest of the end of the demonstration is similar in the opposite case. 
We can thus lower bound the dynamics of $q_k$ (resp. $y_j$) during that time-interval:
\[
\begin{array}{ll}
\dot q_k&=f_k(Q_k(s))-f_k(q_k)\\
&>f_k(Q_k(s_0+\eta/2)) - f_k(q_k) \\
\textrm{and} \\
\dot y_j&> (\beta_j(s_0+\eta/2,y_j) - D) y_j \\
&= (\beta_j(s_0+\eta/2,y_j) - \beta_j(s_0+\eta/2,Y_j(s_0+\eta/2))) y_j 
\end{array}
\]

	Now the convergence of $q_k$ to $Q_k(s_0)$ (resp. $y_j$ to $Y_j(s_0)$) is defined as
% , and will show that it leads to a contradiction : 
% 	Since we assumed the convergence of $q_k$ we have that:
	\begin{equation} \label{conv_q}
	\forall \epsilon>0, \exists t^q>0,  \forall t>t^q, |q_k(t)-Q_k(s_0)|<\epsilon (\textrm{resp. } |y_j(t)-Y_j(s_0)|<\epsilon )
	\end{equation}
since we can pick $\epsilon$ such that $\epsilon<Q(s_0+\eta/4)-Q(s_0)$
(resp. $\epsilon<\min(Y_j(s_0+\eta/4)-Y_j(s_0),Y_j(s_0)-Y_j(s_0-\eta/4))$), we
then have, for $t>t_q$, that $q_k(t)<Q(s_0+\eta/4)$
(resp. $Y_j(s_0-\eta/4)<y_j(t)<Y_j(s_0+\eta/4)$).  Taking our $t^s$ larger
than the 
corresponding $t^q+A(\eta)$, we then have for all time $t\,\in\,[t^s-A(\eta),t^s]$
\[
\begin{array}{ll}
\dot q_k&> f_k(Q_k(s_0+\eta/2)) - \left(f_k(Q_k(s_0)+\epsilon)\right) \\
&>f_k(Q_k(s_0+\eta/2)) - f_k(Q_k(s_0+\eta/4))=C^q(\eta)>0\\
\textrm{and} \\
\dot y_j  &>  ( \beta_j(s_0+\eta/2,Y_j(s_0)+\epsilon) - \beta_j(s_0+\eta/2,Y_j(s_0+\eta/2)) ) y_j\\
&>( \beta_j(s_0+\eta/2,Y_j(s_0+\eta/4)) - \beta_j(s_0+\eta/2,Y_j(s_0+\eta/2)) ) y_j\\
&>  ( \beta_j(s_0+\eta/2,Y_j(s_0+\eta/4)) - \beta_j(s_0+\eta/2,Y_j(s_0+\eta/2)) ) Y_j(s_0-\eta/4)\\
 &=C^y(\eta) >0
\end{array}
\]
with $C^q(\eta)>0$ since $f_k$ is an increasing function of $q_k$, and $C^y(\eta)>0$ since $\beta_j$ is a decreasing function of $y_j$ 
%  and $q_k(t)<Q_k(s_0)+\epsilon<Q_k(s_0)+\eta/4$ when $t>t^q$. 

We then define
\[
C(\eta)=\min\left(C^q(\eta),C^y(\eta)\right)
\]

If we then choose $\epsilon$ such that $\epsilon < \frac{C(\eta) \cdot
  A(\eta)}{2}$, then $|q_k(t^s) - q_k(t^s-A(\eta))| > 2 \epsilon$
(resp. $|y_j(t^s) - y_j(t^s-A(\eta))| > 2 \epsilon$), and we see that the
increase of $q_k$ (resp. $y_j$) makes it eventually get out of the
$\epsilon$-interval around $Q_k(s_0)$ (resp. $Y_j(s_0)$). This is a
contradiction, so that implication $iii => i$ (resp. $v => i$) holds. \\

Alternatively, if for some $t^s$ we have $s(t^s) \leq s_0-\eta$, then we can upper bound the dynamics of $q_k$ (resp. $y_j$) during the $[t^s-A(\eta),t^s]$ time-interval:
\[
\begin{array}{ll}
\dot q_k&<f_k(Q_k(s_0-\eta/2)) - f_k(q_k) \\
\textrm{and} \\
\dot y_j&< (\beta_j(s_0-\eta/2,y_j) - D) y_j \\
&= (\beta_j(s_0-\eta/2,y_j) - \beta_j(s_0-\eta/2,Y_j(s_0-\eta/2))) y_j 
\end{array}
\]
and the same arguments hold, with 
\[
\begin{array}{ll}
\dot q_k< f_k(Q_k(s_0-\eta/2)) - f_k(Q_k(s_0)-\eta/4)=C(\eta)<0 \\
\textrm{and} \\
\dot y_j  <  ( \beta_j(s_0-\eta/2,Y_j(s_0-\eta/4)) - \beta_j(s_0-\eta/2,Y_j(s_0-\eta/2)) ) Y_j(s_0+\eta/4)\\
 =C(\eta) <0
\end{array}
\]
and finally $\epsilon < \frac{-C(\eta) \cdot A(\eta)}{2}$ which causes the contradiction.
\end{proof}

%%%%%%%%%%%%%%%%%%%%%%%%%%%%%%%%%%%%%%%%
% EQUILIBRIA, s_k^\star
%%%%%%%%%%%%%%%%%%%%%%%%%%%%%%%%%%%%%%%%
\subsection{The equilibria correspond to the substrate subsistence concentrations}

In this section we present the equilibria of the generalized competition model (\ref{model_norm}). 
The first equilibrium of this model corresponds to the extinction of all the microorganisms species:
\[
\begin{array}{l}
E_0 = (s_{in}   \ ,\    0,\hdots,0    \ ,\     0,\hdots,0     \ ,\    0,\hdots,0   \ ,\     Q_1(s_{in}),\hdots,Q_N(s_{in}) )
\end{array}
\]
This equilibrium is globally attractive if the input substrate concentration
$s_{in}$ is not high enough for the species' growth to compensate their
withdrawal of the chemostat by the output flow $D$, that is if $\forall i,
\alpha_i(s_{in}) \leq D$ and $\forall j, \beta_j(s_{in},0) \leq D$ and
$\forall k, \gamma_k(Q_k(s_{in})) \leq D$ (proof of this result is easy and we
omit it; for getting a clear idea of the demonstration, see \cite{ArmMcG1980}
and \cite{SmiWal1995} for the Monly and Q-only cases). We suppose that we are
not in this situation through the following hypothesis: 
\begin{hypothesis} \label{hypo:Sin} 
We assume that one of the following condition is satisfied:
\begin{itemize}
\item
$\exists i, \alpha_i(s_{in}) > D$
\item
$\exists j, \beta_j(s_{in},0) > D $
\item
$\exists k, \gamma_k(Q_k(s_{in})) > D$
\end{itemize}
\end{hypothesis}
This guarantees that, at least for one of the families of species, there exists some index $i$, $j$, $k$ and some associated unique $s^{x\star}_i, s^{y\star}_j, s^{z\star}_k < s_{in}$ (denoted "subsistence concentration") such that 
\[
\begin{array}{l}
	\alpha_i(s^{x\star}_i) = D \\
	\beta_j(s^{y\star}_j, Y_j(s^{y\star}_j)) = D \quad \textrm{with} \quad s^{y\star}_j + Y_j(s^{y\star}_j) = s_{in}\\
	\gamma_k(Q_k(s^{z\star}_k)) = D
\end{array}
\]
Note that in the C-model, there exists an infinity of $s \in [S^y_j(0),s_{in})$ verifying $\beta_j(s, Y_j(s)) = D$. 
The value $s^{y\star}_j$ is then the substrate concentration required for
having species $j$ remaining alone in the chemostat at equilibrium. It has to
satisfy $s^{y\star}_j + Y_j(s^{y\star}_j) = s_{in}$ because of
(\ref{eq:dot-M}) that imposes $M=s+y_j=s_{in}$ at equilibrium.  

We number these species such that 
\begin{equation} \label{eq:nx}
\begin{array}{l}
	0< s^{x\star} = s^{x\star}_1 < s^{x\star}_2 < \hdots < s^{x\star}_{n_x}\leq s_{in} \\
	0< S^y_1(0) < S^y_2(0) < \hdots < S^y_{n_y}(0)\leq s_{in} \\
	0< s^{z\star} = s^{z\star}_1 < s^{z\star}_2 <... < s^{z\star}_{n_z}\leq s_{in} \\
	\textrm{with } \forall (i,j,k) \in (\{1,...,n_x\},\{1,...,n_y\},\{1,...,n_z\}),\, s^{x\star}_i \neq S^y_j(0) \neq s^{z\star}_k
\end{array}
\end{equation}
where $n_x$, $n_y$ and $n_z$ are the number or free bacteria, attached
bacteria and phytoplankton species having a subsistence concentration smaller
than $s_{in}$ for the given $D$; all other species cannot be positive at
equilibrium. Hypothesis \ref{hypo:Sin} implies that at least one of $n_x$,
$n_y$ and $n_z$ is non-zero.  
We denote $s^{x\star}$ and $s^{z\star}$ the lowest M- and Q- substrate
subsistence concentrations. We also denote $s^{y\star}$ the substrate
concentration that there would be at equilibrium if there were only attached
species in the chemostat (see \cite{GroMazRap2005}); since it needs to satisfy
(\ref{eq:dot-M}), it requires  
\[
s^{y\star} + \sum_{j=1}^{n_y} Y_j(s^{y\star}) = s_{in}
\]
Though the sum of $Y_j(s^{y\star})$ spans all the relevant indices, some
species might have $Y_j(s^{y\star})=0$ because they have
$s^{y\star}<S^y_j(0)<s_{in}$. If some $n_x$, $n_y$ or $n_z$ is zero, we set
the corresponding  $s^{x\star}$,  $s^{y\star}$ or $s^{z\star}$ to $s_{in}$
because none of the species from their family can survive at a substrate
concentration lower than $s_{in}$, which is the higher admissible
concentration.  

In the previous competitive exclusion studies
\cite{ArmMcG1980,SmiWal1995,GroMazRap2005} these quantities were of primer
importance, as they directed the result of competition. Here we show that the
competition outcome is strongly linked to 
\[
s^\star = \min(s^{x\star},s^{y\star}, s^{z\star})
\]
which is the lowest of all subsistence concentrations. Hypothesis \ref{hypo:Sin} implies that $s^\star<s_{in}$.

We do not consider the case where two subsistence concentrations are equal,
because we suppose that the biological parameters of each species are
different. In his broad historical review about competitive exclusion
\cite{Har1960} Hardin wrote: "no two things or processes in a real world are
precisely equal. In a competition for substrate, no difference in growth rate
or subsistence quota can be so slight as to be neglected".  
\begin{hypothesis}\label{Hyp:hardin}
$\forall (i,j,k) \in (\{1,...,n_x\},\{1,...,n_y\},\{1,...,n_z\}),\, s^{x\star}_i \neq S^y_j(0) \neq s^{z\star}_k$
\end{hypothesis}

\begin{figure}[htp]
	\begin{center}
	\includegraphics[width=4in]{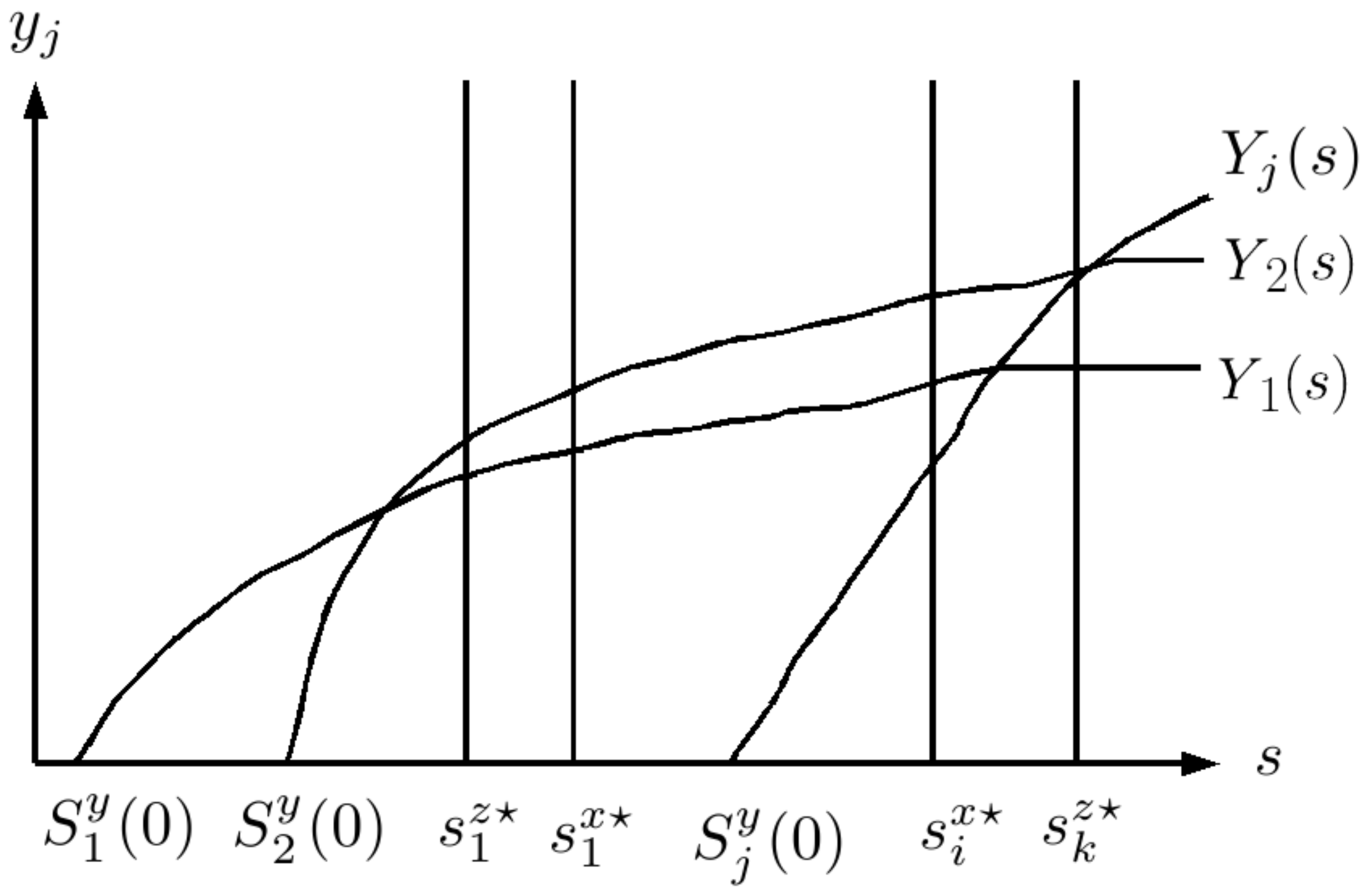}\\
	\caption{Subsistence concentrations of the M-model ($s^{x\star}_i$)
          and Q-model ($s^{z\star}_k$) species, and equilibrium biomass
          $Y_j(s)$ of the attached species, which enable these species to have
          a growth rate equal to the dilution rate $D$, and thus to be at
          equilibrium. We see that M- and Q-model species cannot coexist at
          equilibrium because they have only one fixed subsistence
          concentration, and $s$ cannot be simultaneously equal to several of
          these concentrations. On the contrary, attached species can coexist
          with others at equilibrium because they can have a growth equal to
          $D$ for any $s \in [S^y_j(0),s_{in})$, by adjusting their biomass
          concentration to $Y_j(s)$ (see definition (\ref{Yj})) 
% 	(attraction represented by \emph{arrows})
	}
	\label{fig:s_star}
	\end{center}
\end{figure}
 The subsistence concentrations and $Y_j(s)$ functions are presented in Figure
 \ref{fig:s_star}. In this figure, we see that no free species model can
 coexist at equilibrium because $s$ cannot simultaneously be equal to
 $s^{x\star}_i$ and $s^{z\star}_k$. On the contrary, attached species
 verifying Hypothesis \ref{hypo:Sin} can support different $s$ value at
 equilibrium (between $S^y_j(0)$ and $s_{in}$), so that there exist equilibria
 where one M- or D-model species coexist with one or several attached species
 (see Figure \ref{fig:s_star} for a graphical explanation). On those
 equilibria, only the attached species verifying  
\begin{equation} \label{y_can_co}
S^y_j(0) < s^{x\star}_i \quad (\textrm{resp. } S^y_j(0) < s^{z\star}_k)
\end{equation}
can coexist as they can be at equilibrium at the subsistence concentration of the M-model (resp. D-model) species, by having a biomass equal to $Y_j(s^{x\star}_i)$ (resp. $Y_j(s^{x\star}_k)$).  

For these considerations, we can enunciate the following proposition which does not need to be proved:
\begin{lemma}\label{prop:compliance}
For a given $s_0$ substrate concentration, we have
\begin{itemize}
\item $x_i$ is $s_0$-compliant if $s_i^{x*}=s_0$;
\item $y_j$ is $s_0$-compliant is $S_j(0)<s_0$;
\item $z_k$ is $s_0$-compliant if $s_k^{z*}=s_0$
\end{itemize}
\end{lemma}
It ensues that, for the corresponding C-species we have
\[
Y_j(s_0)>0
\]
which means that they can be at positive equilibrium under dilution rate $D$ and substrate concentration $s_0$. Thus, the $s^\star$-compliant species are:
\begin{itemize}
\item only the $s^\star$-compliant C-species, if $s^\star=s^{y\star}$
\item $x_1$ and all the $s^\star$-compliant C-species, if $s^\star=s^{x\star}$
\item $z_1$ and all the $s^\star$-compliant C-species, if $s^\star=s^{z\star}$
\end{itemize}
We now present all these equilibria and their stability in M-, C- and Q-only substrate competitions.

\subsection{M-only equilibria}\label{sub:M-only}
\[ 
\begin{array}{l}
E^x_i = (s^{x\star}_i   \ ,\    0,\hdots,x_i^\star,\hdots,0    \ ,\     0,\hdots,0   \ ,\    0,\hdots,0 \ , \ Q_1(s^{x\star}_i),\hdots,Q_{N_z}(s^{x\star}_i)   )  \\
\textrm{with } \  x_i^\star = s_{in} - s^{x\star}_i
\end{array}
\]
each of these M-only equilibria corresponds to the winning of competition by
free bacteria species $i$; such an equilibrium only exists for
$i\,\in\,\{0,\cdots,n_x\}$ (all other species cannot survive at a substrate
level lower than $S_{in}$ fot the given $D$). In a competition between several
free bacteria, eqilibrium $E^x_1$ (with lowest substrate subsistence
concentration $s^{x\star}_1$) is asymptotically globally stable, while all the
others are unstable \cite{ArmMcG1980}. 

\subsection{Q-only equilibria}
\[ 
\begin{array}{l}
E^z_k = (s^{z\star}_k   \ ,\    0,\hdots,0    \ ,\     0,\hdots,0  \ ,\    0,\hdots,z_k^\star,\hdots,0  \ ,\    Q_1(s^{z\star}_k),\hdots,Q_{N_z}(s^{z\star}_k)   )  \\
\textrm{with } \  z_k^\star = \frac{s_{in} - s^{z\star}_k}{Q_k(s^{z\star}_k)}
\end{array}
\]
Similarly to M-only equilibria, each of these phytoplankton only equilibria
correspond to the winning of competition by phytoplankton species $k$; such an
equilibrium only exists for $k\,\in\,\{0,\cdots,n_z\}$. In a competition
between several phytoplankton species, equilibrium $E^z_1$ (with lowest
substrate subsistence concentration $s^{z\star}_1$) is asymptotically globally
stable, while all the others are unstable \cite{SmiWal1995}. 

\subsection{C-only equilibria}\label{sec:equilibria}
We denominate $G$ a subset of $\{1,\cdots,n_y\}$ representing any C-species
coexistence. For example if we want to speak about species 1, 5 and 7
coexistence, then we use $G=\{1,5,7\}$. We then define $E^y_G$ the equilibrium
where these species coexist. It is composed by 
\begin{itemize}
\item $s^{y\star}_G$ such that $s^{y\star}_G + \sum_{j \in G} Y_j(s^{y\star}_G)=s_{in}$ because of (\ref{eq:dot-M})
\item $\forall j \in G, y_j=Y_j(s^{y\star})$
\item for any other $j$, $y_j=0$
\item $\forall i \in \{1,\cdots,N_x\}, x_i=0$
\item $\forall k \in \{1,\cdots,N_z\}, z_k=0$
\item $\forall k \in \{1,\cdots,N_z\}), q_k=Q_k(s^{y\star}_G)$
\end{itemize} 
there exist many $E^y_G$ equilibria, corresponding to all the possible $G$
subset. The globally asymptotically stable equilibrium of a competition with
only attached species is given by the choice $G=\{1,\cdots,n_y\}$
\cite{GroMazRap2005}. 
Note that some of the $G$ species can have a null biomass on these equilibria,
as $Y_j(s^{y\star}_G)$ might be null for some $j \in G$. Therefore $E_{G_1}^y$
and $E_{G_2}^y$ with $G_1\neq G_2$ are not necessarily different.  

We must here introduce a technical hypothesis which will be useful later to prove hyperbolicity of the equilibria.
\begin{hypothesis} \label{hyp:technique}
For all $G$ and all $j$ : $S_j(0) \neq s^{y\star}_G$
\end{hypothesis}

\subsection{Coexistence equilibria}\label{sub:coexistence}
As previously said in this section, there also exist equilibria where one of
the free species coexist with several $s^{x\star}_i$- or
$s^{z\star}_k$-compliant  attached bacterial species. 
For a coexistence with free bacteria species we denote them $E^{(x,y)\star}_{i,G}$. 
%with $G$ a subset of $\{1,\cdots,N_y\}$ such that any $j$ species in $G$ is $s^x_i$-compliant
They are composed of:
\begin{itemize}
\item $s=s^{x\star}_i$
\item $\forall j \in G, y_j=Y_j(s^{x\star}_i)$
\item for any other $j$, $y_j=0$
\item $\forall l \neq i, x_l=0$
\item $\forall k \in \{1,\cdots,N_z\}, z_k=0$
\item $\forall k \in \{1,\cdots,N_z\}, q_k=Q_k(s^{x\star}_i)$
\item $x_i=s_{in}-s^{x\star}_i-\sum_{j \in G} Y_j(s^{x\star}_i)$ (this value will be denoted $\bar x_i^G$)
\end{itemize} 
Similarly, for a coexistence with phytoplankton species we denote them $E^{(z,y)\star}_{k,G}$ 
%with $G$ a subset of $\{1,\cdots,N_y\}$ such that any $j$ species in $G$ is $s^z_k$-compliant
. They are composed of:
\begin{itemize}
\item $s=s^{z\star}_k$ 
\item $\forall j \in G, y_j=Y_j(s^{z\star}_k)$
\item for any other $j$, $y_j=0$
\item $\forall i \in \{1,\cdots,N_x\}, x_i=0$
\item $\forall l \neq k, z_l=0$
\item $\forall l \in \{1,\cdots,N_z\}, q_l=Q_l(s^{z\star}_k)$
\item $z_k=\frac{s_{in}-s^{z\star}_k-\sum_{j \in G} Y_j(s^{z\star}_k)}{Q_k(s^{z\star}_k)}$ (this value will be denoted $\bar z_k^G$) 
\end{itemize} 
To our knowledge, these equilibria have never been studied until now.

Note that some of these equilibria might be reduntant with M- or Q-only
equilibria, if all the C-species represented by $G$ are not $s^{x\star}_i$- or
$s^{z\star}_k$-compliant. Note also that all those equilibria do not
necessarily exist in the non-negative orthant. Indeed, $\bar x_i^G$ and $\bar
z_k^G$ can be negative, depending on $s_{in}$ and on the substrate subsistence
concentrations. These equilibria with negative components will not be studied
any further since we only consider initial conditions in the positive orthant,
which is invariant. In the sequel, we will denote $E$ an equilibrium of
(\ref{model_norm}) which belongs to an unspecified class.  

We will now show that if $s=s^\star$ at equilibrium, there exists a positive equilibrium containing all $s^\star$-compliant species.
\begin{lemma}$\;$
\begin{itemize}
\item
If $s^\star=s^{x\star}$ then $E^{(x,y)\star}_{1,\{1,\cdots,n_y\}}$ is in the positive orthant.
\item
If $s^\star=s^{y\star}$ then $E^{y\star}_{\{1,\cdots,n_y\}}$ is in the positive orthant.
\item
If $s^\star=s^{z\star}$ then $E^{(z,y)\star}_{1,\{1,\cdots,n_y\}}$ is in the positive orthant.
\end{itemize}
\end{lemma}

\begin{proof}
\begin{itemize}
\item
If $s^\star=s^{x\star}$, then all $y_j=Y_j(s^{x\star})\geq 0$ at equilibrium
and $s^{y\star}+\sum_{j=1}^{n_y} Y_j(s^{y\star})=s_{in}$ implies that
$s^{x\star}+\sum_{j=1}^{n_y} Y_j(s^{x\star})<s_{in}$ since
$s^{x\star}<s^{y\star}$ and $Y_j(s)$ is non-decreasing. It directly follows
that $x_1^\star=s_{in}-s^{x\star}-\sum_{j}^{n_y} Y_j(s^{x\star})>0$.   %
                                %satisfies $s^{y\star}+\sum_{j=1}^{n_y}
                                %Y_j(s^{y\star})=s_{in}$ and, since
                                %\Leftrightarrow s^{x\star}+\sum_{j=1}^{N_y}
                                %Y_j(s^{x\star})$, so that
                                %$x_i^\star=s_{in}-s^{x\star}_i-\sum_{j \in G}
                                %Y_j(s^{x\star}_i)>0$. 
\item
If $s^\star=s^{y\star}$, then all $x_i$ and $z_k$ are zero at equilibrium and all  $y_j=Y_j(s^{y\star})\geq 0$
\item
If $s^\star=s^{z\star}$, then all $y_j=Y_j(s^{z\star})\geq 0$ at equilibrium
and $s^{y\star}+\sum_{j=1}^{n_y} Y_j(s^{y\star})=s_{in}$ implies that
$s^{z\star}+\sum_{j=1}^{n_y} Y_j(s^{z\star})<s_{in}$ since
$s^{z\star}<s^{y\star}$. It directly follows that
$z_1^\star=\frac{s_{in}-s^{z\star}-\sum_j^{n_z}
  Y_j(s^{z\star})}{Q_k(s^{z\star})}>0$. 
\end{itemize}
\end{proof}

We call $E^\star$ the equilibrium with all $s^\star$-compliant species
remaining in the chemostat, while all the others are excluded. Depending on
the species subsistence concentrations, $E^\star$ can be one of the previously
presented equilibria: 
\begin{itemize}
\item if $s^\star=s^{y\star}$ then $E^\star = E^{y\star}_{\{1,\cdots,n_y\}}$: only the $s^\star$-compliant C-species remain in the chemostat.
\item if $s^\star=s^{x\star}$ then $E^\star =
  E^{(x,y)\star}_{1,\{1,\cdots,n_y\}}$: the best free bacteria species (lowest
  $s^{x\star}_i$) remains in the chemostat with all the
  $s^{x\star}_i$-compliant C-species. 
\item if $s^\star=s^{z\star}$ then $E^\star =
  E^{(z,y)\star}_{1,\{1,\cdots,n_y\}}$: the best phytoplankton species (lowest
  $s^{z\star}_k$) remains in the chemostat with all the
  $s^{z\star}_k$-compliant C-species. 
\end{itemize}
In the next section, we present an important global stability result for this equilibrium.

%%%%%%%%%%%%%%%%%%%%%%%%%%%%%%%%%%%%%%%%
% DEMONSTRATION
%%%%%%%%%%%%%%%%%%%%%%%%%%%%%%%%%%%%%%%%
\section{Statement and demonstration of the Main Theorem: competitive
  exclusion or coexistence in the generalized competition model} 

This theorem states that all the $s^\star$-compliant species (those who can be
at equilibrium with substrate subistence concentration $s^\star$, which is the
lowest of all $s^{x\star}.s^{y\star},s^{z\star}$) coexist in the chemostat at
equilibrium, while all the others are excluded.

%%%%%%%%%%%%%%%%%%%%%%%%%%%%%%%%%%%%%%%%
% MAIN THEOREM
%%%%%%%%%%%%%%%%%%%%%%%%%%%%%%%%%%%%%%%%

\begin{main} \label{main}
	In the generalized competition model (\ref{model_norm}), if Hypotheses \ref{hyp_func_x}--\ref{hyp:technique} hold, 
then all the solutions of the system, having $x_i(0), y_j(0), z_k(0)>0$ for
all $s^*$-compliant species, converge asymptotically towards equilibrium
$E^\star$.  
\end{main}

%%%%%%%%%%%%%%%%%%%%%%%%%%%%%%%%%%%%%%%%
% STRUCTURE OF THE PROOF
%%%%%%%%%%%%%%%%%%%%%%%%%%%%%%%%%%%%%%%%
\emph{Structure of the proof:} In a first step we reduce system
(\ref{model_norm}) to the mass balance surface. Then we present a non
decreasing lower bound $L(t)$ for $s(t)$, and use it to demonstrate that $s$
converges towards $s^\star$. Finally only the $s^*$-compliant species have a
large enough substrate concentration to remain in the chemostat, so that all
other M-, Q-, and C-model species are washed out. The final step consists in
showing that the convergence result that we showed on the mass-balance surface
can be extended to the whole non-negative orthant.   

\begin{remark}
It is not restrictive to consider $x_i(0), y_j(0), z_k(0)>0$ for the solutions
of the system since species with null initial condition can be ignored, so
that we can then consider a smaller dimensional system.  %(see Hypothesis
                                %\ref{hyp:positivity} and Lemma
                                %\ref{lem:positivity}).  
\end{remark}

%%%%%%%%%%%%%%%%%%%%%%%%%%%%%%%%%%%%%%%%
% MASS BALANCE
%%%%%%%%%%%%%%%%%%%%%%%%%%%%%%%%%%%%%%%%
\subsection{Step 1: we consider the system on the mass balance surface and in the region where $q_k\,\in\,(Q_k^0,Q_k^m)$ for all $k\,\in\{1,...,N_z\}$}
\label{sec_MBE}
Lemma \ref{lem:qibound} indicates that $q_k$ reaches $(Q_k^0,Q_k^m)$ in finite time, and
% Through Lemma \ref{qbound}, we have that all $q_k\,\in\,[Q_k^0,Q_k^m)$ in finite time, so that we can limit our proof to that region.
in (\ref{eq:dot-M}) we showed that the total concentration of intra and extracellular substrate in the chemostat $M$ converges to $s_{in}$.

We denote "$\Sigma$", the generalized competition model (\ref{model_norm}) on the mass balance surface defined by
\begin{equation} \label{MBE_eq}
M = s + \sum_{i=1}^{N_x} x_i + \sum_{j=1}^{N_y} y_j + \sum_{k=1}^{N_z} q_k z_k = s_{in}
\end{equation}

For the remainder of the demonstration we will study system $\Sigma$, and we will
later show that its asymptotic convergence towards an equilibrium has the same
behaviour as the initial model (\ref{model_norm}). While studying system
$\Sigma$, we will however retain all the states of the original system and the
expressions of the equilibria; $\Sigma$ is then defined by the addition of the
invariant constraint (\ref{MBE_eq}).  

%We define the equilibria of $\Sigma$ similarly to what was done
%for (\ref{model_norm}):
%\[
%\tilde E_0 = (0,\hdots,0    \ ,\     0,\hdots,0    \ ,\     Q_1(s_{in}),\hdots,Q_N(s_{in})   \ ,\    0,\hdots,0)
%\]
%\[ 
%\tilde E^x_i = (0,\hdots,x_i^\star,\hdots,0    \ ,\     0,\hdots,0   \ ,\    Q_1(s^{x\star}_i),\hdots,Q_{N_z}(s^{x\star}_i)   \ ,\    0,\hdots,0)  \\
%\textrm{with } \  x_i^\star = s_{in} - s^{x\star}_i
%\]
%\[
%\vdots
%\]
%And finally $\tilde E^\star$ the equilibrium with all $s^\star$-compliant species remaining in the chemostat, while all the others are excluded.

%%%%%%%%%%%%%%%%%%%%%%%%%%%%%%%%%%%%%%%%
% DEFINITION OF L
%%%%%%%%%%%%%%%%%%%%%%%%%%%%%%%%%%%%%%%%
\subsection{Step 2: we propose a non decreasing lower bound $L(t)$ for $s$}
\label{sec_def-L}

The main obstacle for the demonstration of the Main Theorem was the
possibility that $s$ would repeatedly be lower than $s^\star$ and repeatedly
be higher than $s_n^\star=\max(s^{x*}_{n_x},s^{z*}_{n_z})$, which would generate an oscillating behaviour. In
order to eliminate this possibility we build a non decreasing lower bound for
$s$, which converges towards $s_1^\star$. We now present such a lower bound,
which will be used to show that $s$ converges to $s^\star$ in the next sections.

\begin{lemma}
In system $\Sigma$
\[
L(t)=\min \left( \min_k(S^z_k(q_k(t))), \min_j(S^y_j(y_j(t)), s^\star, s(t)  \right)
\]
is a non decreasing lower bound for $s$%:
% \[
% \forall t, L(t) \leq s(t) \quad \textrm{ and } \quad \frac{\partial}{\partial t} L(t) \geq 0
% \]
\end{lemma}

\begin{proof}
We know that the right derivative of $L$ is the derivative of one of the
function which realizes the minimum. In four cases we show that this right
derivative is non negative.
% We consider three cases and show that $L(t)$ is a non decreasing lower bound for $s$ in each one

% \begin{itemize}
% \item Case 1: $\min_i(S^z_k(q_k)) \leq s$ and $\min_i(S^z_k(q_k)) \leq s_1^\star$ \\
% Since $S^z_k(q_k)$ has been shown to go towards $s$ for any given $s$ (see
% (\ref{eq:sign})), we then have that if $\min_i(S^z_k(q_k(t))) \leq s(t)$,
% then $L(t) = \min_i(S^z_k(q_k(t)))$ is non decreasing. 
% 
% \item Case 2: $s \leq \min_i(S^z_k(q_k))$ and $s \leq s_1^\star$ \\
% Since we are examining the dynamics of $s$ for system $\Sigma$ ({\it i.e.} on the mass
% balance equilibrium manifold),  we replace $s_{in}$ by $s + \sum_i
% q_kx_i$:
% \[
% 	\dot s = \sum_i \left( D q_k - \rho_k(s) \right) x_i
% \]
% which is equivalent to, from the definition (\ref{Qi}) of $Q_k(s)$ and the definition of $f_k(q_k)$ :
% \[
% 	\dot s = \sum_i \left( Dq_k - \gamma_k(Q_k(s))Q_k(s) \right) x_i
% \]
% Then for all $i$, $s \leq S^z_k(q_k)$ gives us $Q_k(s) \leq q_k$, and $s
% \leq s_1^\star$ gives us $\gamma_k(Q_k(s)) \leq \gamma_k(Q_k(s_k^\star)) =
% D$, so that  
% \[
% \dot s \geq 0
% \]
% In this case $L(t) = s(t)$ is a lower bound for $s$, and it is non decreasing.
%  
% \item Case 3: $s_1^\star \leq s$ and $s_1^\star \leq \min_i(S^z_k(q_k))$ \\
% Here, $L(t) = s_1^\star$ is a constant lower bound for $s$.
% 
% \end{itemize}

\begin{itemize}
\item Case 1: If $S^z_k(q_k(t))$ realizes the minimum then its derivative is non negative, because $S^z_k(q_k)$ goes towards $s$ (see (\ref{eq:sign})).

\item Case 2: If $S^y_j(y_j(t))$ realizes the minimum then its derivative is non negative, because $S^y_j(s)$ goes towards $s$ (see (\ref{eq:sign_y})).

\item Case 3: If $s(t)$ realizes the minimum then we examine its dynamics $\dot s$ for system $\Sigma$ ({\it i.e.} on the mass
balance equilibrium manifold). We replace $s_{in}$ by $s + \sum_i x_i + \sum_j y_j + \sum_k q_k z_k$:
\[
	\dot s = \sum_i (D - \alpha_i(s)) x_i + \sum_j (D - \beta_j(s,y_j)) y_j + \sum_k \left( D q_k - \rho_k(s) \right) z_k 
\]
which is equivalent to, from the definition (\ref{Qi}) of $Q_k(s)$ :
\[
	\dot s = \sum_i (D - \alpha_i(s)) x_i + \sum_j (D - \beta_j(s,y_j)) y_j + \sum_k \left( Dq_k - \gamma_k(Q_k(s))Q_k(s) \right) z_k
\]
Then 
\begin{itemize}
\item
for all $i$, $s \leq s^\star$ gives us $\alpha_i(s) \leq D$, so that the first sum is non negative;
\item 
for all $j$, $s \leq S^y_j(y_j))$ gives us $\beta_j(s,y_j) \leq \beta_j(S^y_j(y_j),y_j) = D)$, so that the second sum is non negative;
\item
for all $k$, $s \leq S^z_k(q_k)$ gives us $Q_k(s) \leq q_k$, and $s \leq
s^\star$ gives us $\gamma_k(Q_k(s)) \leq \gamma_k(Q_k(s_k^{z\star})) = D$ so
that the third sum is also non negative. 
\end{itemize}
Finally we obtain 
\[
\dot s \geq 0
\]
 
\item Case 4: If $s^\star$ realizes the minimum, we know that its right derivative is null and thus non negative.

\end{itemize}

\end{proof}

%%%%%%%%%%%%%%%%%%%%%%%%%%%%%%%%%%%%%%%%
% CONVERGENCE OF s TO s_1^\star
%%%%%%%%%%%%%%%%%%%%%%%%%%%%%%%%%%%%%%%%
\subsection{Step 3: we demonstrate that $s$ converges towards $s^\star$}
\label{sec_s-conv-s1}

\begin{lemma} \label{s_conv_s1}
In system $\Sigma$
\[
\limt s(t) = s^\star
\]
\end{lemma}

\begin{proof}

We  first show, by contradiction, that the substrate concentration $s(t)$ cannot
converge towards any constant value other than $s^\star$. Suppose the reverse hypothesis, {\it i.e.} $\limt s(t)=\bar s\neq s^\star$. Through
Lemma \ref{conv_s-qi}, we then have that $\limt q_k(t)=Q_k(\bar s)$ and $\limt y_j(t)=Y_j(\bar s)$. 

If $\bar s<s^\star$, 
\begin{itemize}
\item
$\alpha_i(\bar s) < D$ for all $i$ so that all $x_i$ go to $0$
\item
$\gamma_k(Q_k(\bar s))<D$ for all $k$ implies that all $z_k$ go to $0$
\end{itemize}
So that we have a contradiction with mass balance equilibrium (\ref{MBE_eq}),
as the total substrate (in the medium + in the biomasses) at equilibrium $\bar
s+ \sum_{j=1}^{N_y} Y_j(\bar s)$ will be lower than $s^{y\star}+
\sum_{j=1}^{N_y} Y_j(s^{y\star}) = s_{in}$. 

If $\bar s>s^\star$ we must consider three cases:
\begin{itemize}
\item 
if $s^\star=s^{x\star}_1$ then $\alpha_i(\bar s)>D$ implies that $x_1$
diverges to $+\infty$, which is in contradiction with the boundedness shown in
(\ref{eq:s-low-Mm}).  
\item 
if $s^\star=s^{z\star}_1$ then $\gamma_1(Q_1(\bar s))>D$ implies that $z_1$
diverges to $+\infty$, which is in contradiction with the boundedness shown in
(\ref{xibound}).  
\item
if $s^\star=s^{y\star}$ then we have a contradiction with mass balance
equilibrium (\ref{MBE_eq}), because $\bar s+ \sum_{j=1}^{N_y} Y_j(\bar s)$
will be higher than $s^{y\star}+ \sum_{j=1}^{N_y} Y_j(s^{y\star}) \leq
s_{in}$. 
\end{itemize}

Hence the impossibility of convergence of $s$ towards any
$\bar s$ other that $s^\star$ is proven.

We now demonstrate the lemma by contradiction. We assume that 
\[
s \textrm{ does not converge towards } s^\star
\]
which, from the previous remark means that $s$ does not converge to any constant value. 
\begin{remark}
As $s$ does not converge towards $s^\star$, we know that the $q_k$ do not converge towards $Q_k(s^\star)$ (see Lemma \ref{conv_s-qi})
\end{remark}
We consider two cases, which both lead to a contradiction, on the basis of a
reasoning which is close to the demonstration developed to prove Lemma
\ref{conv_s-qi}. \\ 

$\bullet$ Case $a$: $L$ attains $s^\star$ in finite time \\
In Appendix \ref{app_4a} we show that a contradiction occurs. \\
{\it Idea :} $s$ cannot stay higher than $s^\star$ without converging to
$s^\star$, because this would cause $x_1$ or $z_1$ to diverge, or $s+
\sum_{j=1}^{N_y} Y_j(s)$ to be always higher than $s_{in}$ without converging
to $s_{in}$ .

$\bullet$ Case $b$: $L$ never attains $s^\star$ \\
See Appendix \ref{app_4b}. \\
{\it Idea :} If $s$ did not converge to $s^\star$, the non decrease of $L$ and its attraction by $s$ would cause it to reach $s^\star$.

In both cases we found a contradiction, so that the proof of Lemma \ref{s_conv_s1} is complete.

\end{proof}

%%%%%%%%%%%%%%%%%%%%%%%%%%%%%%%%%%%%%%%%
% END OF DEMONSTRATION
%%%%%%%%%%%%%%%%%%%%%%%%%%%%%%%%%%%%%%%%
\subsection{Step 4: all the $s^\star$-compliant species remain in the chemostat, while the others are excluded}

In this section we show that, as $s$ converges towards $s^\star$ in model
$\Sigma$, all the free species with substrate subsistence concentration higher
than $s^\star$ are washed out of the chemostat because their growth
$\alpha_i(s)$ or $\gamma_k(q_k)$ cannot stay high enough to compensate the
output dilution rate $D$. Finally, all the $s^\star$-compliant species able to
be at equilibrium with a substrate concentration $s^\star$ remain in the
chemostat. 

\begin{lemma} \label{conv-x1}
In system $\Sigma$ all the solutions with positive initial conditions for the $s^*$-compliant species converge to $E^\star$.
\end{lemma}

\begin{proof}
For all the $x_i$ and $z_k$ species such that $\alpha_i(s^\star)<D$ and
$\gamma_k(Q_k(s^\star))<D$, it is straigthforward that the convergence of $s$
to $s^\star$ will cause their biomass to converge to $0$. If
$s^\star=s^{y\star}$, then this is true for all the free species.  

For all the $s^\star$-compliant C-species, we have from Lemma \ref{conv_s-qi}
that their biomass will tend to $Y_j(s^\star)$, which is positive for the
$s^\star$-compliant species and null for all the others. 

Finally, if $s^\star=s^{x\star}$ or $s^\star=s^{z\star}$, then we have through
the mass balance equilibrium (\ref{MBE_eq}) that the free species whose
subsistence concentration is $s^\star$ will have its biomass converge to
$s_{in} - s^\star - \sum_{j=1}^{N^y_{z1}} Y_j(s^{\star})$: all the substrate
which is not present in the medium or in the attached biomasses is used by the
best M- or Q-competitor. 
\end{proof}

\subsection{Step 5: convergence of the solutions for model $\Sigma$ implies
  convergence for model (\ref{model_norm})}

%***************** Je n'ai pas encore actualis\'e le Step 5 ****************** \\

In order to extend the convergence result to the full model and thus prove our Main Theorem, we 
apply a classical theorem for asymptotically autonomous system \cite{Thieme1992,SmiWal1995}. 

\begin{lemma}\label{lemma:final}
All solutions of system (\ref{model_norm}) with positive initial conditions
for the $s^*$-compliant species converge to $E^*$ defined in section
\ref{sec:equilibria}.  
\end{lemma}

% The proof will be made by induction on the dimension of the system. Supposing
% that competitive exclusion has been shown for model (\ref{droop}) in dimension
% $N-1$, we will show that  it is still valid in dimension $N$, based on the
% convergence result for model $\Sigma$ in dimension $2N$ (where the $s$
% dynamics have been omitted).

\begin{remark}
While, up to here, we simply had considered $\Sigma$ as the same system as
(\ref{model_norm}), in the same dimension, except that it was restricted to
(\ref{MBE_eq}), we will now equivalently explicitely include (\ref{MBE_eq})
into system (\ref{model_norm}) to obtain $\Sigma$ in the form of a system that
has one dimension less than (\ref{model_norm}) by omitting the $s$
coordinate. Since both representations of $\Sigma$ are equivalent, the
previously proven stability results are still valid in the new representation,
with the exception that convergence takes place towards equilibria directly
derived from these presented in sections
\ref{sub:M-only}-\ref{sub:coexistence} by omitting the $s$ coordinate. These
new equilibria are differentiated from the original ones by adding a \
$\tilde{}$\ \ , so that an arbitrary equilibrium is denoted $\tilde
E$. \end{remark}

\begin{proof}
System $\Sigma$ can be written as follows : 
\begin{equation}
\label{model_sigma}
\begin{array}{l}
\left\{
\begin{split}
	\dot x_i &= \left(\alpha_i\left(s_{in}-\sum_{l=1}^{N_x} x_l - \sum_{m=1}^{N_y} y_m - \sum_{r=1}^{N_z} q_r z_r  \right) - D\right) x_i \\
	\dot y_j &= \left(\beta_j\left(s_{in}-\sum_{l=1}^{N_x} x_l - \sum_{m=1}^{N_y} y_m - \sum_{r=1}^{N_z} q_r z_r,y_j\right) - D\right) y_j \\
	\dot z_k &= (\gamma_k(q_k) - D) z_k \\
	\dot q_k &= \rho_k\left(s_{in}-\sum_{l=1}^{N_x} x_l - \sum_{m=1}^{N_y} y_m - \sum_{r=1}^{N_z} q_r z_r  \right)- f_k(q_k) 
\end{split}
\right. 
\\
\textrm{ for } i \in \{1,\cdots,N_x\}, j \in \{1,\cdots,N_y\},k \in \{1,\cdots,N_z\} \\
\end{array}
\end{equation}
where the $s$ state has been removed compared to (\ref{model_norm}). In order to recover model (\ref{model_norm}), we should add the equation
\[
\dot M=D(s_{in}-M)
\]
which we interconnect with (\ref{model_sigma}) by replacing every $s_{in}$ in
(\ref{model_sigma}) with $M$. It is this interconnection that we will now
study.  
%We define the equilibria of $\Sigma$ similarly to what was done
%for (\ref{model_norm}):
%\[
%\tilde E_0 = (Q_1(s_{in}),\hdots,Q_k(s_{in}),\hdots,Q_N(s_{in})   \quad,\quad   0,\hdots,0,\hdots,0)
%\]
%\[ 
%\tilde E_i = (Q_1(s_k^\star),\hdots,Q_k(s_k^\star),\hdots,Q_N(s_k^\star)
%\quad,\quad   0,\hdots,x_i^\star,\hdots,0) \mbox{ for }i\leq n
%\]

In the first part of the proof, we will show that every solution of
(\ref{model_norm}) converges to an equilibrium $E$. We will then show by
induction that all the solutions that do not converge to $E^*$ have an initial
condition with some $x_i=0$, $y_j=0$ or $z_k=0$ for some $s^*$-compliant
species. Thus, all the solutions with $x_i\neq 0, y_j\neq 0, z_k\neq 0$ for
the $s^*$-compliant species converge to $E^*$. 

 For that, we will use Theorem F.1 from \cite{SmiWal1995}. We will therefore first compute the  stable manifolds of all equilibria of $\Sigma$:
\begin{itemize}
\item
The stable manifold of $\tilde E^*$ is of dimension $N_x+N_y+2N_z$. It is
constituted of all the initial conditions which verify $x_i(0), y_j(0),
z_k(0)>0$ for $s^*$-compliant species and $x_i(0), y_j(0), z_k(0)\geq 0$ for
all other species, as well as $q_k\geq 0$ for all $k$ (see Lemma
\ref{conv-x1}). 
\item The stable manifold of $\tilde E_0$ is of dimension $N_x-n_x+N_y-n_y+ 2
  N_z - n_z$. It is constitued of all the initial conditions which verify
  $x_1(0)=\hdots=x_{n_x}(0)=0$, $y_1(0)=\hdots=y_{n_y}(0)=0$ and
  $z_1(0)=\hdots=z_{n_z}(0)=0$. The only species that can be present at the
  initial condition are those that cannot survive for the given $D$ and
  $s_{in}$. Indeed, if any $x_i(0)>0$ for $i\leq n_x$ (or similar $y_j(0)>0$
  or $z_k(0)>0$), one can apply Lemma \ref{conv-x1}. to the reduced order
  system containing these species to show that convergence does not take place
  towards $\tilde E_0$. Conversely, any initial condition with
  $x_1(0)=\hdots=x_{n_x}(0)=0$, $y_1(0)=\hdots=y_{n_y}(0)=0$ and
  $z_1(0)=\hdots=z_{n_x}(0)=0$ generates a solution that goes to $\tilde E_0$
  since for the other species we have: 
\begin{itemize}
\renewcommand{\labelitemi}{$\triangleright$}
\item $\dot x_i<(\alpha_i(s_{in})-D)x_i$, with $\alpha_i(s_{in})-D<0$ for all $i > n_x$ because of the definition of $n_x$ presented in (\ref{eq:nx});
\item $\dot y_j<(\beta_j(s_{in},y_j)-D)y_j$, with $\beta_j(s_{in},0)-D<0$ for all $j > n_y$, because of the definition of $n_y$;
\item $\dot z_k <(\gamma_k(Q_k(s_{in}))-D)z_k$, with $\gamma_k(Q_k(s_{in}))-D)<0$ for all $k > n_z$ because of the definition of $n_z$.
\end{itemize}

% \item
% The stable manifold of $\tilde E_2$ is of dimension $2N-1$. It is constitued of all the initial condition which verify $x_1(0)=0$ and $x_2(0)>0$. 
\item
The dimension of the stable manifold of any other $\tilde E$ can be computed
from Lemma \ref{conv-x1}. To an equilibrium $\tilde E$ corresponds a substrate
value $\tilde s$ ($>s^*$ by definition of $s^*$). Lemma \ref{conv-x1}
indicates that solutions of $\Sigma$ converge towards an equilibrium
corresponding to $\tilde s$, if there is no smaller subsistance concentration
corresponding to a species present in the system (for free species) and if all
M-, Q- and C-species that are $\tilde s$-compliant are present in the
corresponding equilibrium. The stable manifold of $\tilde E$  must therefore
be constrained to initial conditions that verify $x_i(0)=0$, $y_j(0)=0$ and
$z_k(0)=0$ for all species that are $s$-compliant for some $s\leq \tilde s$
and that are not positive in $\tilde E$. %Lemma \ref{conv-x1} indeed indicates
                                %that, if some $S$- or $Q$- species has its
                                %$s_i^{x*}$ or $s_k^{z*}$ smaller than $\tilde
                                %s$, convergence of solutions of $\Sigma$ with
                                %the corresponding $x_i(0)>0$ or $z_k(0)>0$
                                %cannot take place with $s$ going to $\tilde
                                %s$, but to a value that is smaller or equal
                                %to $s_i^{x*}$ or $s_k^{z*}$; convergence to
                                %$\tilde s$ is then only possible if all such
                                %$x_i(0)=0$ and $z_k(0)=0$. If now some
                                %$SB$-species $y_j$ is $\tilde s$-compliant,
                                %we have shown that it has to be present in
                                %the equilibrium with $s=\tilde s$ towards
                                %which the solutions converge; all the $\tilde
                                %s$-compliant $SB$-species that are absent
                                %from $\tilde E$ then must satisfy $y_j(0)=0$.
                                %
Having set all these values to zero, it is indeed clear that $\tilde s$ is the
$s^\star$ as defined in Lemma \ref{conv-x1} of the reduced order system
(without the aforementionned $x_i, y_j$ and $z_k$ coordinates). All solutions
defined in Lemma \ref{conv-x1} of this system then converge to $\tilde E$,
which justifies our definition of the stable manifold of $\tilde E$. Its
dimension is $N_x+N_y+2N_z-n_{\tilde E,\tilde s}$, where $n_{\tilde E,\tilde
  s}$ is the number of $s$-compliant species (for some $s\leq \tilde s$) that
are not present in $\tilde E$. 
\end{itemize}
Through Lemma \ref{conv-x1}, we have in fact shown that all solutions of
$\Sigma$ in the non-negative orthant converge to an equilibrium. Indeed, for a
given initial condition, either it belongs to the stable manifold of $\tilde
E_0$ or, eliminating from the system all species that are null at the initial
time necessarily sets it in a form where Lemma \ref{conv-x1} can be applied
(which shows convergence to an equilibrium).    

The dimension of the stable manifold of any equilibrium $E$ will therefore be the one of $\tilde E$ plus $1$.
The hypotheses of Theorem F.1 from \cite{SmiWal1995} are indeed all verified:
\begin{itemize}
\item
The whole system (\ref{model_norm}) is bounded (see section \ref{sec:bounds})
\item
The equilibria of system $\Sigma$ are hyperbolic (see Appendix \ref{sub:C2}-\ref{sub:C6}).
\item
There are no cycles of equilibria in system $\Sigma$. Indeed, if we analyze
the potential transition between two equilibria, both equilibria must belong
to the same face, so that convergence takes place to the one corresponding to
the smallest value of $s$. A potential sequence of equilibria would then be
characterized by a decreasing value of $s$ at each equilibrium, which prevents
it from cycling.  
\end{itemize}
We can then conclude from this theorem  that all solutions of
(\ref{model_norm}) tend to an equilibrium. We are then left with checking to
what equilibrium they tend.  

Before continuing this proof, we need to detail $n_{\tilde E,\tilde s}$. In
the case of $\tilde E=\tilde E^*$ and $\tilde s=s^*$, we have $n_{\tilde
  E,\tilde s}=0$ (by definition, all $s^*$-compliant species are present in
$E^*$ and there is no other species that is compliant for smaller values of
$s$). Otherwise, we necessarily have  $n_{\tilde E,\tilde s}>0$. Indeed, we
know that $\tilde s>s^*$, so that all species present in $E^*$ are compliant
for some $s<\tilde s$; as such, in order to have $n_{\tilde E,\tilde s}=0$,
$\tilde E$ would need to at least contain all species that are present in
$\tilde E^*$. In such a case no $S$ and $Q$ species can be present in $\tilde
E^*$ (otherwise, it could not be present in $\tilde E$ also for a different
value of $s$). Defining $J$ the set of C-species that are present in $E^*$ and
writing (\ref{MBE_eq}) for $E^*$ then yields 
\[
M = s^* + \sum_{j\,\in\,J} Y_j(s^*)= s_{in}
\]
Equality (\ref{MBE_eq}) should also be valid in $\tilde s>s^*$ so that
\[
s_{in} = \tilde s + \sum_{i=1}^{N_x} x_i + \sum_{j=1}^{N_y} y_j + \sum_{k=1}^{N_z} q_k z_k >\tilde s + \sum_{j\,\in\,J} Y_j(\tilde s)>s_{in}
\]
where we have the last inequality (which leads to a contradiction) because
$Y_j(s)$ is an increasing function. We can then conclude that, for all $\tilde
E\neq \tilde E^*$, $n_{\tilde E,\tilde s}>0$, and at least one
$s^\star$-compliant species species must be null.

In order to check to what equilibrium solutions of  (\ref{model_norm}) tend, we use an induction argument, by supposing that our
Main Theorem has been proven up to $N-1$ species, which we use for the
proof for $N$ species. Along with the fact that the stability result is
trivial for $1$ species (classical Monod model, \cite{SmiWal1995}, classical
Droop model, \cite{OyaLan94} and generalized Contois model,
\cite{GroMazRap2005}), this will conclude our proof. 

Let us consider a system of $N$ species with equilibrium $E^*$ as defined
earlier. This equilibrium contains positive species (which are
$s^\star$-compliant) and null species (which are not $s^\star$-compliant).  

Imposing, for one of the not $s^\star$-compliant species, $x_i=0$ (or $y_j=0$
or $z_k=0$) for the initial condition, sets us in the framework where we have
$N-1$ species present in the system. Also, since this species did not belong
to the positive ones in $E^\star$, its absence does not change anything into
which equilibrium is the one corresponding to the smallest subsistance
concentration, which remains $E^\star$. We can then apply the induction
hypothesis, which indicates that all such initial conditions initiate
solutions that converge to $E^\star$ (as long as the $s^\star$-compliant
species have positive initial condition).  

Studying now the equilibrium $E_0$, we know from the beginning of the proof
that its stable manifold is of dimension $N_x-n_x+N_y-n_y+ 2 N_z - n_z+1$. As
was done for $\Sigma$, it is directly apparent that any initial condition with
$x_1(0)=\hdots=x_{n_x}(0)=0$, $y_1(0)=\hdots=y_{n_y}(0)=0$ and
$z_1(0)=\hdots=z_{n_x}(0)=0$ generates a solution that has all species
exponentially go to zero. Finally, the analysis of the $\dot s$ equation shows
that it has the form $\dot s=D(s_{in}-s)-F(t)$ with $F(t)$ exponentially going
to zero so that $s$ goes to $s_{in}$ and all such solutions go to $E_0$.  

We can now consider all the other equilibria. Let an equilibrium $E$
corresponding to a substrate concentration $\tilde s$ ($>s^*$ by
definition). As we have seen in our analysis of $\Sigma$, the stable manifold
of the corresponding $\tilde E$ is of dimension $N_x+N_y+2N_z-n_{\tilde
  E,\tilde s}$, so that the stable manifold of $E$ is of dimension
$N_x+N_y+2N_z-n_{\tilde E,\tilde s}+1$. Let us set ourselfes in the situation
where all  $n_{\tilde E,\tilde s}$ species are set to zero at the initial time
and all others are positive. We can then consider the system with only the
remaining $N_x+N_y+2N_z-n_{\tilde E,\tilde s}$ positive species and the
substrate. We have seen that, in this case, all solutions of the corresponding
reduced order $\Sigma$ go to $\tilde E$ which means that $\tilde s$ is the
``$s^*$'' defined in Lemma \ref{lemma:final} for the reduced order
system. Since the reduced order system contains less than $N$ species because
$n_{\tilde E,\tilde s}>0$, we conclude that all solutions of the full system
(\ref{model_norm}) that have zero initial condition for all $n_{\tilde
  E,\tilde s}$ species and positive values for all $N_x+N_y+2N_z-n_{\tilde
  E,\tilde s}$ others converge to $\tilde E$. We have then exhibited an
invariant manifold of dimension $N_x+N_y+2N_z-n_{\tilde E,\tilde s}+1$ for
which all solutions go to $E$; this corresponds to the predicted dimension of
the stable manifoldof $E$. No solution with some of the $n_{\tilde E,\tilde
  s}$ species positive (among which there is at least on $s^*$-compliant
species) at the initial time can then converge to $E$.

This completes the proof of our Main Theorem since all solutions go to an
equilibrium and we have exhibited the stable manifold of all equilibria other
than $E^*$. These manifolds cannot go into the region where $x_i, y_j$ or
$z_k>0$ for all $s*-compliant$ species because at least one of them is in the
corresponding $n_{\tilde E,\tilde s}$-set. All initial conditions in the
region where $x_i, y_j$ or $z_k>0$ for all $s*-compliant$ species therefore
generate solutions that go to $E^*$.  
\end{proof}

%%%%%%%%%%%%%%%%%%%%%%%%%%%%%%%%%%%%%%%%
% DISCUSSION
%%%%%%%%%%%%%%%%%%%%%%%%%%%%%%%%%%%%%%%%
\section{Discussion}
\subsection{ How $D$ and $s_{in}$ both determine competition outcome}

\label{sec:discussion}

In M- and Q-only competitions, the outcome of competition is mainly determined
by $D$, which fixes the $s^{x\star}_i$ and $s^{z\star}_k$  M- and Q-substrate
subsistence concentrations; the role of $s_{in}$ is to allow the best
competitor (already determined by the value of $D$) to settle the reactor, or
to cause it to be washed out with all the others. On the contrary in C-only
competition, both controls have important roles: $D$ fixes the $Y_j(s)$
functions, while $s_{in}$ determines the equilibrium, where $s^{y\star}+\sum_j
Y_j(s^{y\star}) = s_{in}$. With a low enough $s_{in}$, only few C-species will
settle the chemostat ($s^{y\star}$ being low in this case, there will be few
$s^{y\star}$-compliant species, with non-null $Y_j(s^{y\star})$), whereas a
high enough $s_{in}$ can enable all C-species to coexist. 

Finally, in a mixed competition the dilution rate $D$ fixes all the M- and
Q-substrate subsistence concentrations $s^{x\star}_i$ and $s_k^{z\star}$, as
well as the $Y_j(s)$ functions, while the input substrate concentration
$s_{in}$ selects the species remaining in the reactor, by limiting the
available nutrients, and thus the biomasses present in the reactor at
equilibrium. Figure \ref{fig:discussion} gives an example between three
competitors. 

\begin{figure}[htp]
	\begin{center}
	\includegraphics[width=4in]{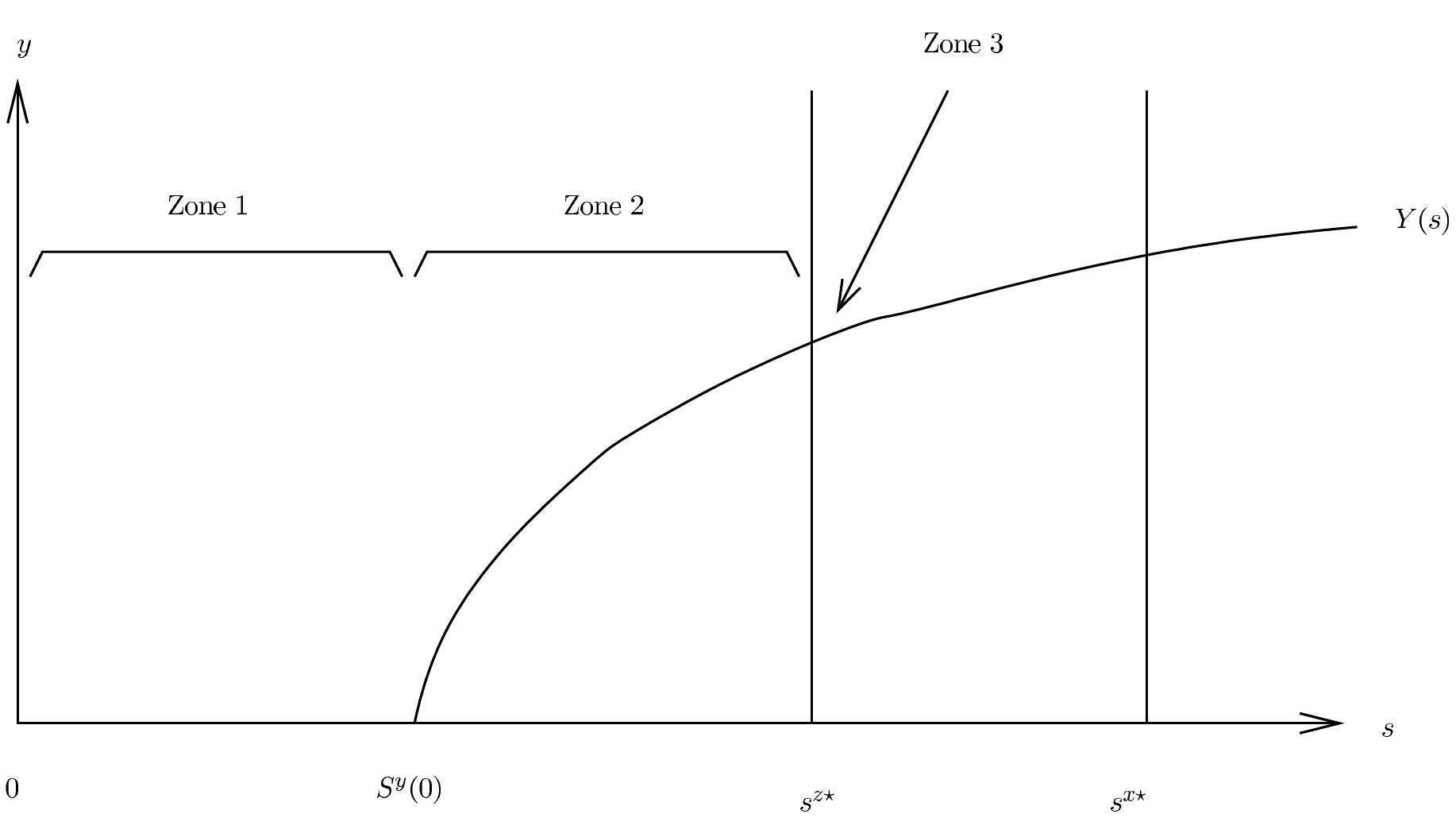}\\
	\caption{Mixed 3 class competition outcome depends both on the
          dilution rate $D$ and input substrate concentration $s_{in}$. The
          solid lines represent the influence of $D$, which fixes the
          subsistence concentrations of one M-model ($s^{x\star}$) and Q-model
          ($s^{z\star}$) species, and the equilibrium biomass $Y(s)$ of one
          C-species. As the free bacteria species has a too high subsistence
          concentration $s^{x\star}>s^{z\star}$ it will be outcompeted and
          excluded. The three numerated zones represent the influence of
          $s_{in}$. {\it Zone 1} ($s_{in} \leq S^y(0)$) : no species remain at
          equilibrium. {\it Zone 2} ($S^y(0)<s_{in} \leq
          s^{z\star}+Y(s^{z\star})$) : only the C-species remains at
          equilibrium. {\it Zone 3} ($s_{in}>s^{z\star}+Y(s^{z\star})$) : the
          attached bacteria and phytoplankton species coexist.  
% 	(attraction represented by \emph{arrows})
	}
	\label{fig:discussion}
	\end{center}
\end{figure}

On this figure the $s^{x\star}$ and $s^{z\star}$ values and the $Y(s)$
function are fixed by $D$. Here $s^{z\star}_1$ is lower than $s^{x\star}_1$,
so that the free bacteria species will be outcompeted and washed out. Then the
value of $s_{in}$ determines wether 
\begin{enumerate}
\item
no species remain at equilibrium
\item
only the attached bacterial species remains at equilibrium, as there is not
enough input substrate to feed both attached bacteria and phytoplankton
species: because $s^{y\star} + Y(s^{y\star}) = s_{in}$ and $s_{in} <
s^{z\star} + Y(s^{z\star})$, we know that $s^{y\star}<s^{z\star}$, so that
$s^\star=s^{y\star}$, and only the C-species remains in the reactor. 
\item
both the attached bacteria and phytoplankton species remain in the chemostat:
here $s_{in}>s^{z\star} + Y(s^{z\star})$ and $s_{in} =
s^{y\star}+Y(s^{y\star})$ give $s^{y\star}>s^{z\star}$, so that
$s^\star=s^{z\star}$ and the phytoplankton species remains in the reactor,
coexisting with the $s^{z\star}$-compliant C-species. 
\end{enumerate}
In this last case $D$ has fixed the $s^{z\star}$ substrate equilibrium value
and the $Y(s)$ function, and at equilibrium the total substrate in the
chemostat, equal to $s_{in}$, will be composed of 
\begin{itemize} 
\item
the substrate in the medium $s^{z\star}$ (which is fixed by $D$ and does not depend on $s_{in}$);
\item
the attached bacterial species internal substrate $Y(s^{z\star})$ (which is also fixed by $D$ only);
\item
the phytoplankton species internal substrate $Q(s^{z\star})) z^\star= s_{in} - Y(s^{z\star}) - s^{z\star}$, which depends on $s_{in}$.
\end{itemize}
By going from left to right in Figure \ref{fig:discussion}, starting with
$s_{in}=0$, it is possible to imagine the input substrate concentration
increase, thus enabeling more and more substrate $s=s_{in}$ at equilibrium
(zone 1). Then in zone 2 the C-species is present at equilibrium, and as
$s_{in}$ increases, more and more biomass $Y(s)$ is present at
equilibrium. Finally $Y(s^{z\star})$ is the maximal biomass for which the
attached species needs less substrate at equilibrium than the phytoplankton
species to have a growth rate equal to $D$. After that it has to coexist with
the phytoplankton species: when $s_{in}$ increases higher than
$s^{z\star}+Y(s^{z\star})$ it enables more and more Q-biomass $z^\star$, while
keeping substrate concentration $s=s^{z\star}$ and C-biomass
$y=Y(s^{z\star})$. 

\subsection{Originiality of the demonstration}

The demonstration explains how the state variables evolve, and its originality
for the study of uniquely phytoplankton (or bacteria) species can be summed up
in three points. \\ 
First, we chose to study the substrate evolution instead of ignoring it after
the classical mass balance equilibrium transformation $s = s_{in} - \sum_i x_i
- \sum_j y_j -\sum_k{q_k z_k}$. \\ 
Then, the definition of the $S^y_j$ and $S^z_k$ functions enabled to gather
most information on the substrate axis:  
instead of having separate information on $1+N_x+N_y+2N_z$ axes we obtained a
one dimensional view on these dynamics (Figure \ref{fig_Qi-Si} and
\ref{fig_Yj-Sj}), where all the $S^y_j(y_j)$ and $S^z_k(q_k)$ go towards
$s$. We have thus turned a complex $1+N_x+N_y+2N_z$ dimensional problem into a
simpler one: "how do $s$ and the $S^y_j(y_j)$ and $S^z_k(q_k)$ behave on the
substrate axis, and what are the consequences for the biomasses?". \\ 
Finally the definition of the non decreasing lower bound $L(t)$ (section
\ref{sec_def-L}) and its convergence towards $s^\star$ (section
\ref{sec_s-conv-s1}) were the last steps for this demonstration to emerge. 
 
Free species pure competitions (with one class of species among Monod or
Droop) for substrate lead to the "survival of the fittest", the fittest being
the species with lowest substrate requirement $s^\star$.  
On the contrary, Contois-only competition lead to a coexistence equilibrium,
because biomass dependence gives attached bacterial species the capability to
remain at equilibrium for different substrate concentrations in the range
$[S^y_j(0),s_{in})$ (see Figure \ref{fig:s_star}). Monod and Droop species are
mutually exclusive, which leads to the pessimization principle of adaptative
dynamics \cite{Diekmann2003} : "mutation and natural selection lead to a
deterioration of the environmental condition, a Verlenderung. We end up with
the worst of all possible environment." On the contrary attached species are
coexistence-compliant thanks to biomass dependence, which nuances the
pessimization principle: "some species could live in worse environments
($s=\min_j(S^y_j(0))$ being the worse one) but if there is enough substrate
for other species, they can coexist." (see Figure \ref{fig:discussion} and
discussion) 

%%%%%%%%%%%%%%%%%%%%%%%%%%%%%%%%%%%%%%%%
% CONCLUSION
%%%%%%%%%%%%%%%%%%%%%%%%%%%%%%%%%%%%%%%%
\section{Conclusion}
In this paper a demonstration was given for the outcome of competition between
phytoplankton and bacteria. Three scenarios are possible, depending both on
the dilution rate $D$ and input substrate concentration $s_{in}$ (see
discussion for precisions): 
\begin{itemize}
\item
only the best free competitor remains in the chemostat;
\item 
only some attached bacterial species coexist at equilibrium;
\item 
a new equilibrium (never studied before) is attained, where the best free competitor coexists with all the $s^\star$-compliant attached bacterial species.
\end{itemize}

Since the introduction of the concept of evolution, with its link to
competitive exclusion \cite{Har1960} and the "paradox of phytoplankton"
\cite{Hutchinson1961}  
modelling has tried to apprehend competition, and to predict or control
it. Our contribution in this framework was to extend the results proven in the
N-species Monod model, N-species Droop model and N-species Contois model,
where the outcome of competition was predicted and explained with mathematical
arguments, accompanied by ecological interpretations. 

An important conclusion in this type of competition is that attached bacteria
are likely to be present in a pure culture of microalgae. This may have very
important consequences on the ecological point of view, since such natural
coexistence between a phytoplanktonic species and attached bacteria may have
lead to co-evolution, where the best association between phytoplankton and
bacteria have been progressively selected.

\appendix

% \section{Convergence (theorem) at mass balance equilibrium}
% \label{app_conv_thm}

% ****** FREDERIC, \'a toi de jouer si tu veux =O) ******

\section{Step 3 - Case $a$: $L$ attains $s^\star$ in finite time}
\label{app_4a}

%%%%%%%%%%%%%%%%%%%%%%%%%%%%%%%%%%%%%%%%
% L ATTAINS s_1^\star
%%%%%%%%%%%%%%%%%%%%%%%%%%%%%%%%%%%%%%%%

In this case 
\begin{itemize}

\item 
if $s^\star = s^{z\star}_1$ 
we consider Figure \ref{fig_L-att-s1} where $L$ attains $s^\star$ after a finite time $t^L$:
\[
	\forall t \geq t^L, L(t) = s_1^\star
\]

\begin{figure}[htp]
\begin{center}
\includegraphics[width=4in]{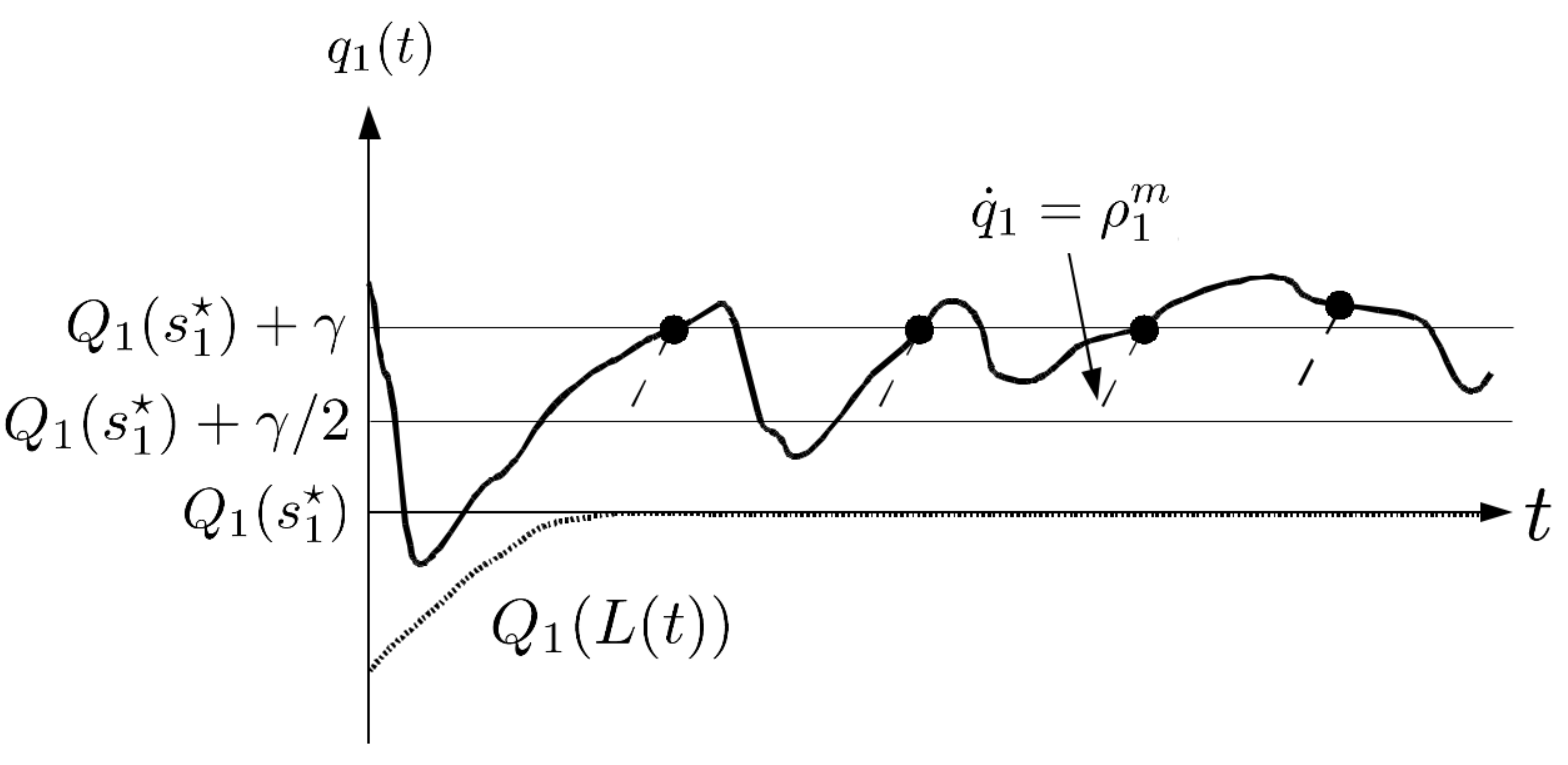}\\
\caption{Visual explanation of the demonstration of Lemma \ref{s_conv_s1} - Case 1: $L$ attains $s^\star$ in finite time $t^L$ (Q-model).
$i)$ $q_1$ is repeatedly higher than $Q_1(s^\star)+\theta$ (\emph{$\bullet$}). 
$ii)$ Because $\dot q_1$ is upper bounded by $\rho^m_1$,  so that $q_1$ is
higher than $Q_1(s^\star)+\theta/2$ during non negligible time intervals
(\emph{dashed lines} represent $\dot q_1 = \rho^m_1$). Thus $z_1$ diverges,
which is a contradiction.} 
\label{fig_L-att-s1}
\end{center}
\end{figure}

% \begin{compactenum}[i)]

% \item
\subsubsection*{Substep 3a.1: after a finite time larger than $t^L$, $q_1$ is repeatedly higher than $Q_1(s^\star) + \theta$.}

Since $\min_i(S^z_k(q_k)) \geq L$, we know that
\[
\forall t>t^L, q_1(t)\geq Q_1(s^\star)
\]
As $s$ does not converge to $s^\star$, we also know from Lemma \ref{conv_s-qi} that $q_1$ does not converge towards $Q_1(s^\star)$:
\[
\exists \theta>0, \forall t>0, \exists t^q>t, |q_1(t^q) - Q_1(s^\star)| > \theta
\]
Those two facts imply that the repeated exits of $q_1(t)$ from the
$\theta$-interval around $Q_1(s^\star)$ take place above $Q_1(s^\star)$ 
for any $t^q > t^L$, so that, in that case, we have $q_1(t^q)>Q_1(s^\star) +
\theta$. In Figure \ref{fig_L-att-s1}, such $t^q$ time instants are
represented by \emph{$\bullet$}. 

% \item
\subsubsection*{Substep 3a.2: $q_1$ is higher than $Q_1(s^\star) + \theta/2$ during non negligible time intervals.}

Since the $q_1$-dynamics are upper bounded with
\[
\dot q_1 \leq \rho^m_1
\]
we know that every time $q_1$ is higher than $Q_1(s^\star) + \theta$, it has
been higher than $Q_1(s^\star) + \theta/2$ during a time interval of minimal
duration $A(\theta) = \frac{\theta}{2 \rho^m_1}$. On Figure
\ref{fig_L-att-s1}, $\dot q_1 = \rho^m_1$ is represented by the \emph{dashed
  lines}. 

% \item
\subsubsection*{Substep 3a.3: then $z_1$ diverges, which is impossible}

From time $t^L$ on, we have that $q_1\geq Q_1(s^\star) \Rightarrow
\gamma_1(q_1) \geq D$, so that $z_1(t)$ is non decreasing. During each of the
time interval where $q_1$ is higher than $Q_1(s^\star) + \theta/2$, the
increase of $z_1$ is lower bounded by  
\[
\dot z_1 = \gamma_1(Q_1(s^\star) + \theta/2) - D = C(\theta) > 0
\]
so that every $t^q$ time we have 
\[
z_1(t^q)-z_1(t^q-A(\theta)) > C(\theta) A(\theta)
\]
As such increases occurs repeatedly, and as $z_1$ is non decreasing, $z_1$
diverges. This is a contradiction because $z_1$ is upper bounded (see
(\ref{xibound})). \\ 

% \end{compactenum}

\item 
if $s^\star = s^{x\star}_1$  then the non convergence of $s$ to $s^\star$, and the fact that $s \geq s^\star$ will cause $s$ to be non negligibly "away"
 from $s^\star$, so that $x_1$ will diverge, causing a contradiction with
 (\ref{eq:s-low-Mm}). This is exactly the same demonstration as above (in the
 case $s^\star=s^{z\star}_1$) without needing the $q_k$ study. 

\item
if $s^\star = s^{y\star}$ then $s+\sum_{j=1}^{N_y} Y_j(s)$ will always be
higher than $s_{in}=s^{y\star}+\sum_{j=1}^{N_y} Y_j(s^{y\star})$ without
converging to $s_{in}$, which is in contradiction with (\ref{MBE_eq}). 

\end{itemize}

\section{Step 3 - Case $b$: $L$ never attains $s^\star$}
\label{app_4b}

%%%%%%%%%%%%%%%%%%%%%%%%%%%%%%%%%%%%%%%%
% L NEVER ATTAINS s_1^\star
%%%%%%%%%%%%%%%%%%%%%%%%%%%%%%%%%%%%%%%%

In this case $L(t)$ converges towards a value $ \hat L \in (0,s^\star]$,
because it is non decreasing and bounded in $[0,s^\star]$, so that 
\[
\forall \epsilon >0, \exists t^L(\epsilon)>0, \forall t>t^L(\epsilon), \arrowvert L(t)-\hat L\arrowvert <\epsilon
\]
We consider the neighborhood of $\hat L$ in Figure \ref{fig_L-no-att-s1}.

\begin{figure}[htp]
\begin{center}
\includegraphics[width=4in]{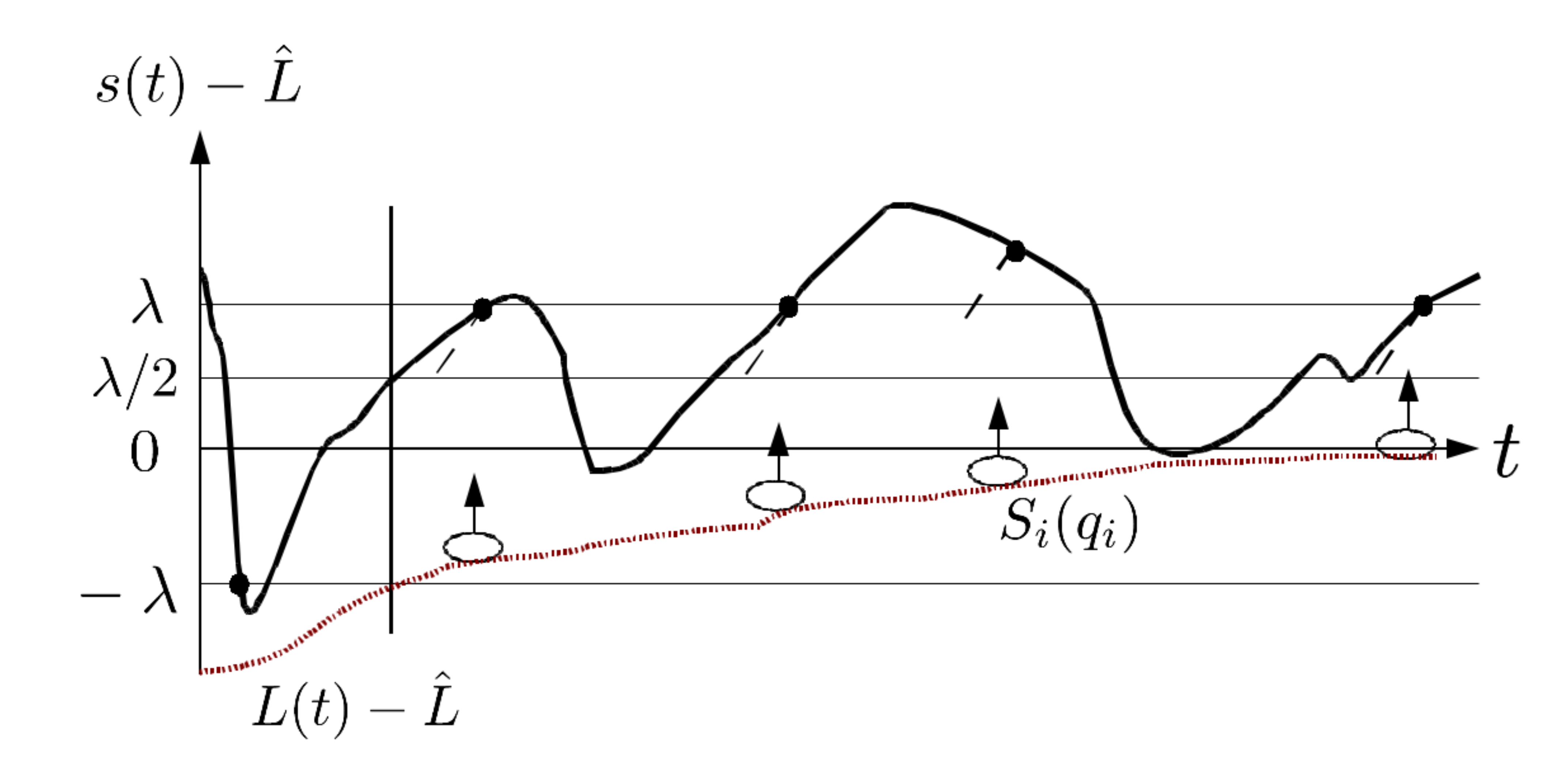}\\
\caption{Visual explanation of the demonstration of Lemma \ref{s_conv_s1} - Case 2: $L$ never attains $s^\star$.
$i)$ $s$ is repeatedly higher than $\hat L + \lambda$ (\emph{$\bullet$}).
$ii)$ $\dot s$ is upper bounded by $D s_{in}$, so that $s$ is higher than
$\hat L + \lambda/2$ during non negligible time intervals (\emph{dashed lines}
represent $\dot s = D s_{in}$) 
$iii)$ during such a time intervals $L = \min_k(S^z_k(q_k))$ (or
$\min_j(S^y_j(y_j))$) is increasing non negligibly towards $s$, so that $L$
cannot both converge towards $\hat L$ and stay lower than $\hat L$ during the
whole time interval: there is a contradiction. 
% the increase of $L = \min_i\{S^z_k(q_k)\}$ towards $s$, and its convergence towards $\hat L$ lead to a contradiction.
% $i)$ $q_1$ is repeatedly higher than $Q_1(s_1^\star)+\lambda$ (\emph{big dots}). 
% $ii)$ Because $\dot q_1$ is upper bounded by $\rho^m_1$,  $q_1$ is higher
% than $Q_1(s_1^\star)+\lambda/2$ during non negligible time intervals
% (\emph{dashed lines} represent $\dot q_1 = \rho^m_1$). Thus $x_1$ diverges,
% which is a contradiction. 
}
\label{fig_L-no-att-s1}
\end{center}
\end{figure}

\subsubsection*{Substep 3b.1: after a finite time, $s$ is repeatedly higher than $\hat L + \lambda$.}

Since, from the beginning of the proof of Lemma \ref{s_conv_s1},  we know that
 $s$ does not
converge to any constant value, hence not to $\hat L$, 
\[
\exists \lambda>0, \forall t>0, \exists t^s>t, |s(t^s) - \hat L| > \lambda
\]

Since $L$ is increasing and converges to $\hat L$, it reaches $\hat L -
\lambda$ in finite time $t^L(\lambda)$. After this finite time, $s$ is higher
than $\hat L + \lambda$ on every $t^s$ time instants, which are represented by
\emph{$\bullet$} in Figure \ref{fig_L-no-att-s1}. 

\subsubsection*{Substep 3b.2: $s$ is higher than $\hat L + \lambda/2$ during non negligible time intervals.}

Because of the boundedness of $\dot s$
\[
\dot s \leq D s_{in}
\]
every time $s$ is higher than $\hat L + \lambda$, it has been higher than
$\hat L + \lambda/2$ during a non negligible time interval of minimal duration
$A(\lambda) = \frac{\lambda}{2 D s_{in}}$. On Figure \ref{fig_L-no-att-s1} the
case $\dot s = D s_{in}$ is represented by \emph{dashed lines}. 

\subsubsection*{Substep 3b.3: $L = \min_k(S^z_k(q_k))$ (or
  $\min_j(S^y_j(y_j))$) is increasing non negligibly towards $s$, so that $L$
  cannot both converge towards $\hat L$ and stay lower than $\hat L$ during
  the whole time interval: there is a contradiction.} 

Like in previous proofs, we are interested in what
happens during the $[t^s - A(\lambda),t^s]$ time-interval, with 
$t^s -A(\lambda)>t^L(\epsilon)$ (for some $\epsilon<\lambda$).
% t^L(\epsilon)\geq t^L(\lambda)$ 
% (for any choice of $\epsilon\leq\lambda$).
Since, during this time-interval, $s(t) > \hat L + \lambda/2$ and
$L < \hat L$, we know that there exists a $k$ such that $L(t^s) = S^z_k(q_k(t^s)) < \hat L$, or a $j$ such that $L(t^s) = S^y_j(y_j(t^s)) < \hat L$. 

For both this step (3b.3) we choose to first only present arguments for the
case $L(t^s) = \min_k(S^z_k(q_k))$; almost similar arguments for the case
$L(t^s) = \min_j(S^y_j(y_j))$ will then be briefly presented.

%  Since, during this time-interval, $s(t) > \hat L + \lambda/2$ and
% $L > \hat L - \epsilon$, we have that $L$ is the lowest $S^z_k(q_k)$, but that 
% the species whose $S^z_k(q_k)$ is minimal could change. 

% For $L$ not to attain $\hat L$, there must exist a species $k$ for every
% such interval $[t^s - A(\lambda),t^s]$ that is such that $S^z_k(q_k)$ stays
% below $\hat L$ during the whole interval. Otherwise, with $s\geq \hat
% L+\lambda/2$ and all $S^z_k(q_k)$ reaching $\hat L$ before $t^s$, $L$ would
% have reached $\hat L$ (note that if, for some j, $S_j(q_j(t^s -
% A(\lambda)))\geq \hat L$, it stays that way during the whole interval due to
% the attractiveness of $s$ for $S_j(q_j)$).
\begin{itemize}
\item 
if $L(t^s) = \min_k(S^z_k(q_k))$, then 
during the whole considered time-interval, as $S^z_k(q_k)$ was increasing, we know that 
\begin{equation} \label{L-Sk}
\hat L - \epsilon < L \leq S^z_k(q_k) \leq S^z_k(q_k(t^s)) < \hat L
\end{equation}
so that $Q_k(\hat L-\epsilon) < q_k(t) < Q_k(\hat L)$. 
For the $k$ species, the dynamics of $q_k$ can then be lower bounded:
\[
\dot q_k \geq \rho_k(\hat L+\lambda/2) - f_k(Q_k(\hat L))
\]
and then
\[
\dot q_k \geq \rho_k(\hat L+\lambda/2) - \rho_k(\hat L) = G_k(\lambda)
\]
positive, so that the increase of $q_k$ during the $[t^s - A(\lambda),t^s]$ time-interval is also lower bounded:
\[
q_k(t^s) - q_k(t^s - A(\lambda)) \geq G_k(\lambda) A(\lambda) = H_k(\lambda)
\]
% As $Q_k=S^z_k^{-1}$ is a K-Lipschitz function, the corresponding increase of $S^z_k(q_k)$ is also lower bounded:
% % (by $\frac{1}{K} \beta_k(\lambda)$):
Since $Q_k = S^{z^{-1}}_k$ is locally Lipschitz with constant $K$ (because $f'_k > 0$), we have
\[
\begin{array}{ll}
q_k(t^s)-q_k(t^s-A(\lambda)) &= Q_k(S^z_k(q_k(t^s)))-Q_k(S^z_k(q_k(t^s-A(\lambda))))\\
&< K \left[ S^z_k(q_k(t^s)) - S^z_k(q_k(t^s-A(\lambda))) \right]
\end{array}
\]
so that the corresponding increase of $S^z_k(q_k)$ is lower bounded with
\[
\begin{array}{ll}
S^z_k(q_k(t^s)) - S^z_k(q_k(t^s - A(\lambda))) 
% & = \frac{\partial S^z_k}{\partial t} A(\lambda) \\
% & = \frac{\partial S^z_k}{\partial q_k} \dot q_k A(\lambda) \\
% & \geq \frac{\partial S^z_k}{\partial q_k} \beta_k(\lambda) \\
& \geq \frac{1}{K} H_k(\lambda)
\end{array}
\]
and then
\[
S^z_k(q_k(t^s - A(\lambda))) < \hat L - \frac{1}{K} H_k(\lambda)
\]
which implies the same higher bound for $L$: 
% at time $t^s - A(\lambda)$:
\[
L(t^s - A(\lambda)) < \hat L - \frac{1}{K} H_k(\lambda)
\]
% This higher boundedness of $L(t^s - A(\lambda))$ is not compatible with its convergence to $\hat L$, which can be seen 
By choosing $\epsilon < \frac{1}{K} H_k(\lambda)$, this inequality is contradictory with (\ref{L-Sk}) so that Case 2 is not possible

\item 
if $L(t^s) = \min_j(S^y_j(y_j))$, then 
the same arguments can be developped for the $j$ species, with a lower bound $G_j(\lambda)$ on the $y_j$ dynamics:
\[
G_j(\lambda) = \beta_j(\hat L+\lambda/2,Y_j(\hat L)) - \beta_j(\hat L,Y_j(\hat L)) > 0
\]
and then an increase of variable $y_j$ at least equal to $H_j(\lambda) = G_j(\lambda) A(\lambda)$ 
followed by a non negligible increase of $L$
\[
L(t^s - A(\lambda)) < \hat L - \frac{1}{K} H_j(\lambda)
\]
because $Y_j$ is locally Lipschitz. Finally a contradiction 
also occurs when $\epsilon < \frac{1}{K} H_j(\lambda)$:
\[
L(t^s - A(\lambda)) < \hat L - \epsilon
\] 
\end{itemize}

\section{Computation of system $\Sigma$ Jacobian Matrix and eigenvalues for all the equilibria}
\label{jacob}

Computation of the Jacobian Matrix of system $\Sigma$, with $s=s_{in}-\sum_{i=1}^{N_x} x_i-\sum_{j=1}^{N_y} y_j-\sum_{k=1}^{N_z} q_k z_k$.
\[
	\left(
	\begin{array}{cccc}
		J^{xx} & J^{xy} & J^{xz} & J^{xq} \\
		J^{yx} & J^{yy} & J^{yz} & J^{yq} \\
		J^{zx} & J^{zy} & J^{zz} & J^{zq} \\
		J^{qx} & J^{qy} & J^{qz} & J^{qq} 
	\end{array}
	\right)
\]
where
\[
\begin{array} {l}
J^{xx}_{ii} = \alpha_i(s)-D -\frac{\partial \alpha_i}{\partial s} x_i
\quad \textrm{and} \quad \forall l \neq i, J^{xx}_{il} = -\frac{\partial \alpha_i}{\partial s} x_i \\
J^{xy}_{ij} = -\frac{\partial \alpha_i}{\partial s} x_i \\
J^{xz}_{ik} = -\frac{\partial \alpha_i}{\partial s} x_i q_k \\
J^{xq}_{ik} = -\frac{\partial \alpha_i}{\partial s} x_i z_k \\
\textrm{and}\\
J^{yx}_{ji} = -\frac{\partial \beta_j}{\partial s} y_j \\
J^{yy}_{jj} = \beta_j(s,y_j)-D + \frac{\partial \beta_j}{\partial y_j}y_j
-\frac{\partial \beta_j}{\partial s} y_j \quad \textrm{and} \quad \forall l
\neq j, J^{yy}_{jl} = -\frac{\partial \beta_j}{\partial s} y_j \\ 
J^{yz}_{jk} = -\frac{\partial \beta_j}{\partial s} y_j q_k \\
J^{yq}_{jk} = -\frac{\partial \beta_j}{\partial s} y_j z_k \\
\textrm{and}\\
J^{zx}_{ki} = 0 \\
J^{zy}_{kj} = 0 \\ 
J^{zz}_{kk} = \gamma_k(q_k)-D \quad \textrm{and} \quad \forall l \neq k, J^{zz}_{kl} = 0 \\
J^{zq}_{kk} = \frac{\partial \gamma_k}{\partial q_k} z_k \quad \textrm{and} \quad \forall l \neq k, J^{zq}_{kl} = 0 \\
\textrm{and}\\
J^{qx}_{ki} = -\frac{\partial \rho_k}{\partial s} \\
J^{qy}_{kj} = -\frac{\partial \rho_k}{\partial s} \\ 
J^{qz}_{kl} = -\frac{\partial \rho_k}{\partial s} q_l \\
J^{qq}_{kk} = -\frac{\partial \rho_k}{\partial s} z_k - \frac{\partial
  f_k}{\partial q_k} \quad \textrm{and} \quad \forall l \neq k, J^{qq}_{kl} =
-\frac{\partial \rho_k}{\partial s} z_l \\ 
\end{array}
\]

Fortunately for eigenvalue computations, at equilibria the null biomasses will simplify the matrix:
\begin{itemize}
\item
when $x_i=0$, then the whole $i^{th}$ line gives eigenvalue $\alpha_i(s)-D$ (denoted "$x_i$-eigenvalue") and can be deleted, as well as the $i^{th} column$;
\item
when $y_j=0$ then the whole $N_x+j^{th}$ line gives eigenvalue
$\beta_j(s,y_j)-D$ (denoted "$y_j$-eigenvalue") and can be deleted, as well as
the $N_x+j^{th}$ corresponding column; 
\item
when $z_k=0$ then the whole $N_x+N_y+k^{th}$ line gives eigenvalue
$\gamma_k(q_k)-D$ (denoted "$z_k$-eigenvalue") and can be deleted, as well as
the $N_x+N_y+k^{th}$ column; in a second step, the whole $N_x+N_y+N_z+k^{th}$
column can also be deleted and gives eigenvalue $\frac{-\partial f_k}{\partial
  q_k}$ (denoted "$q_k$-eigenvalue"), as well as the $N_x+N_y+N_z+k^{th}$
line. 
\end{itemize}

% so that for equilibrium $\tilde E_0$ the Jacobian matrix is described by
% \[
% \begin{array} {l}
% J^{xx}_{ii} = \alpha_i(s_{in})-D \quad \textrm{and} \quad \forall j \neq i, J^{xx}_{ij} = 0 \\
% J^{xy}_{ij} = 0 \\
% J^{xz}_{ik} = 0 \\
% J^{xq}_{ik} = 0 \\
% \textrm{and}\\
% J^{yx}_{ji} = 0 \\
% J^{yy}_{jj} = \beta_j(s_{in},0)-D \quad \textrm{and} \quad \forall j \neq i, J^{yy}_{ji} = 0 \\
% J^{yz}_{jk} = 0 \\
% J^{yq}_{jk} = 0 \\
% \textrm{and}\\
% J^{zx}_{ki} = 0 \\
% J^{zy}_{kj} = 0 \\ 
% J^{zz}_{kk} = \gamma_k(Q_k(s_{in}))-D \quad \textrm{and} \quad \forall i \neq k, J^{zz}_{ki} = 0 \\
% J^{zq}_{kk} = 0 \quad \textrm{and} \quad \forall i \neq k, J^{zq}_{ki} = 0 \\
% \textrm{and}\\
% J^{qx}_{ki} = -\frac{\partial \rho_k}{\partial s} \\
% J^{qy}_{kj} = -\frac{\partial \rho_k}{\partial s} \\ 
% J^{qz}_{ki} = -\frac{\partial \rho_k}{\partial s} Q_i(s_{in}) \\
% J^{qq}_{kk} = - \frac{\partial f_k}{\partial q_k} \quad \textrm{and} \quad \forall i \neq k, J^{qq}_{ki} = 0 \\
% \end{array}
% \]

\subsection{Complete washout equilibrium}\label{sub:C1}
With this in hand, we see that for equilibrium $\tilde E_0$ ($x_i=y_j=z_k=0$)
the Jacobian matrix is triangular, so that the eigenvalues lay on the
diagonal. They are: 
% is triangular, so that the eigenvalues lay on the diagonal. They are: 
\begin{itemize}
\item
$\alpha_i(s_{in})-D$
\item 
$\beta_j(s_{in},0)-D$
\item
$\gamma_k(Q_k(s_{in}))-D$
\item
$- \frac{\partial f_k}{\partial q_k} \quad \textrm{(negatives)}$
\end{itemize}
We denote $n_x$, $n_y$, $n_z$ the number of M-, C- and Q- species verifying
the inequalities of Hypothesis \ref{hypo:Sin}, and thus having the possibility
to be at equilibrium with a positive biomass, under controls $D$ and
$s_{in}$. Each of these species has a positive corresponding eigenvalue on
this equilibrium, so that  
equilibirum $\tilde E_0$ has $n_x+n_y+n_z$ positive eigenvalues, and $N_x-n_x+N_y-n_y+ 2 N_z - n_z$ negative eigenvalues.

\subsection{M-only equilibria}\label{sub:C2}
For equilibrium $E^x_i$ we get all the previously cited x-,y-, z-and q-eigenvalues: 
\begin{itemize}
\item
$\alpha_l(s^{x\star}_i)-D$ whose signs are the same as $sign(s^{x\star}_i-s^{x\star}_l)$
\item
$\beta_j(s^{x\star}_i,0)-D$ which are positive if the $j^{th}$ species is $s^{x\star}_i$-compliant, or negative else;
\item
$\gamma_k(Q_k(s^{x\star}_i))-D$ whose signs are the same as $sign(s^{x\star}_i-s^{z\star}_k)$
\item
$\frac{-\partial f_k}{\partial q_k}$ which are all negative
\end{itemize}
and the remaining eigenvalue corresponds to the positive $x_i$-only dynamics:
\[
\dot x_i=(\alpha_i(s_{in}-x_i)-D)x_i
\]
which yields the eigenvalue $-\frac{\partial \alpha_i}{\partial s} x_i^\star$ for free bacteria species $i$.
Each free species with a substrate subsistence concentration $s^{x\star}_l$ or $s^{z\star}_k$ lower than $s^{x\star}_i$ gives a positive eigenvalue.
Among all the $E^x_i$ equilibria, only $E^x_1$ is stable if and only if
$s^\star=s_1^{x\star}<s_1^{z\star}$, and if all the C-species are not
$s^{x\star}_1$-compliant. %We have demonstrated previously that in this case
                          %this equilibrium is globally asymptotically
                          %stable. 

\subsection{Q-only equilibria}\label{sub:C3}
For Equilibrium $E^z_k$ we get all the 
\begin{itemize}
\item 
$x$-eigenvalues whose signs are the sign of $sign(s^{z\star}_k-s^{x\star}_i)$;
\item
$y$-eigenvalues: as previously, $y$-eigenvalues are positive if the corresponding C-species is $s^{z\star}_k$-compliant and negative else;
\item
$z_l$-eigenvalues  whose signs are the sign of $sign(s^{z\star}_k-s^{z\star}_l)$;
\item
$q_l$-eigenvalues for all $l \neq k$ (negative);
\end{itemize}
and the remaining eigenvalues correspond to the positive $(z_k,q_k)$-only dynamics:
\[
\left\{\begin{array}{lll}
\dot z_k& =& (\gamma_k(q_k) - D) z_k \\
\dot q_k& =& \rho_k(s_{in}-q_kz_k) - f_k(q_k) 
\end{array}\right.
\]
and we obtain the following resulting matrix:
\[
	\left(
	\begin{array}{cccc}
		0 & \frac{\partial \gamma_k}{\partial q_k} z_k \\
		-\frac{\partial \rho_k}{\partial s} q_k & -\frac{\partial \rho_k}{\partial s} z_k-\frac{\partial f_k}{\partial q_k} 
	\end{array}
	\right)
\]
which has negative trace and positive determinant, so that its two eigenvalues are real negative.
Just like before, each free species with a substrate subsistence concentration
$s^{x\star}_i$ or $s^{z\star}_l$ lower than $s^{z\star}_k$ gives a positive
eigenvalue. 
Among all the $E^z_k$ equilibria, only $E^z_1$ is stable if and only if
$s^\star=s_1^{z\star}<s_1^{x\star}$, and if all the C-species are not
$s^{z\star}_1$-compliant. %In this case we have seen in the paper that this
                          %equilibrium is globally asymptotically stable. 

\subsection{C-only equilibria}\label{sub:C4}
Now let us consider the $E^y_G$ equilibria for which all $j \in G$ (where $G$
represents a subset of $\{1,\hdots,N_y\}$) C-species coexist in the chemostat
under substrate concentration $s^{y\star}_G$, while all the free species are
washed out. $s^{y\star}_G$ is defined by $s^{y\star}_G + \sum_{j \in G}
Y_j(s^{y\star}_G) = s_{in}$. Note that some of the $G$ species can have a null
biomass on these equilibria, as $Y_j(s^{y\star}_G)$ might be null for some $j
\in G$. 

This gives all the 
\begin{itemize}
\item 
$x$-eigenvalues whose sign are the same as the signs of $s^{x\star}_i-s^{y\star}_G$;
\item 
$z$-eigenvalues whose sign are the same as the signs of $s^{z\star}_k-s^{y\star}_G$;
\item
$q$-eigenvalues (negative).
\end{itemize}
All the $y_j$ species who are not included in $G$ give negative eigenvalues if
they are not $s^{y\star}_G$-compliant, and positive eigenvalues else; their
eigenvalues cannot be null because of technical hypothesis
\ref{hyp:technique}. All the $y_j$ species who are included in $G$ but have a
null biomass $Y_j(s^{y\star}_G)$ on the $E^y_G$ equilibrium give negative
eigenvalues. Now let us study the remaining matrix $J^{yy}_G$ which is
composed of all the $j \in G$ lines of $J^{yy}$, for which
$Y_j(s^{y\star}_G)>0$, and thus $\beta_j(s^{y\star}_G, Y_j(s^{y\star}_G))=D$: 
\[
\dot y_j=\left(\beta_j(s-\sum_ly_l,y_j)-D\right)y_j
\]
which yields the Jacobian matrix:
\[
J^{yy}_G = 
	\left(
	\begin{array}{ccccc}
-a_1-b_1 & \hdots & -a_1 & \hdots & -a_1\\
\vdots & \ddots & \vdots &  & \vdots \\
-a_j &  \hdots & -a_j-b_j & \hdots & -a_j \\
\vdots &  & \vdots & \ddots & \vdots \\
-a_n &  \hdots & -a_n & \hdots & -a_n-b_n \\
	\end{array}
	\right)
\]
with $a_j = \frac{\partial \beta_j}{\partial s}Y_j(s^{y\star}_G)>0$ and $b_j = -\frac{\partial \beta_j}{\partial y_j}Y_j(s^{y\star}_G)>0$.

Let us show that this matrix has only real negative eigenvalues, by using the
definition of an eigenvalue $\lambda=(A + B i)$, where $A \in \mathbb{R}$ is
the real part and $B \in \mathbb{R}$ the imaginary part: 
\begin{equation} \label{eq:def_eigenvalue}
J^{yy}_G \cdot 
\left(
\begin{array}{l}
y_1 \\ \vdots \\ y_n
\end{array}
\right)
= (A+Bi)
\left(
\begin{array}{l}
y_1 \\ \vdots \\ y_n
\end{array}
\right)
\end{equation}
We obtain $n$ equations:
\[
-b_j y_j - a_j \sum_l y_l = (A+Bi) y_j
\]
and thus
\begin{equation} \label{eq:ABi}
(A+Bi+b_j) y_j  = - a_j \sum_l y_l
\end{equation}
If we have $A+Bi+b_j = 0$ for some $j$, then $B=0$ and $A=-b_j<0$ so that we have a negative eigenvalue. \\
Else, isolating $y_j$ yields
\[
y_j=\frac{- a_j \sum_l y_l}{b_j+A+Bi}
\]
Summing over $j$, we obtain
\[
\sum_j y_j=\sum_j\left(\frac{- a_j \sum_l y_l}{b_j+A+Bi}\right)
\]
Now if $\sum_j y_j=0$, since some $y_j$ must be different of $0$,
(\ref{eq:ABi}) yields, for that $j$, that $A+Bi+b_j = 0$ so that again $B=0$
and $A=-b_j<0$. \\ 
Else, simplifying the sums of $y_l$ and $y_j$, this yields
\[
\begin{array}{lll}
1&=&\sum_j\left(\frac{- a_j}{b_j+A+Bi}\right)\\
&=&\sum_j\left(\frac{- a_j(b_j+A-Bi)}{(b_j+A)^2+B^2}\right)\\
&=&\sum_j\left(\frac{- a_j(b_j+A)}{(b_j+A)^2+B^2}\right)+i\sum_j\left(\frac{a_jB)}{(b_j+A)^2+B^2}\right)
\end{array}
\]
Since the left-hand-side is real, the imaginary part of the right-hand side
must be zero, which imposes $B=0$. For thr right-hand-side to be positive, at
least one of the $b_j+A$ must be negative, which translates into $\min_j
(b_j+A)<0$ and  
\[
A<-\min_j b_j<0
\]
We conclude from this that all eigenvalues of this matrix are real negative.

Finally, an $E^y_G$ equilibrium is stable if and only if all the C-species not
contained in $G$ are not $s^{y\star}_G$-compliant (this is equivalent to
saying that $s^{y\star}_G = s^{y\star}, \textrm{ with } s^{y\star}=
s^{y\star}_{\{1,\hdots,N_y\}}$), and if $s^\star=s^{y\star}$. %We showed in
                                %this paper that in this case this equilibrium
                                %is globally asymptotically stable. 

\subsection{M-coexistive equilibria}\label{sub:C5}
In this section we consider equilibria $E^{(x,y)}_{i,G}$ where free bacteria
species $x_i$ coexists with the C-species in $G$, a subset of
$\{1,\hdots,N_y\}$, under substrate concentration $s^{x\star}_i$. 

We obtain here all the 
\begin{itemize}
\item
$x_l$-eigenvalues ($l \neq i$) whose signs are the signs of $s^{x\star}_i-s^{x\star}_l$;
\item
$z$-eigenvalues whose sign is the sign of $s^{x\star}_i-s^{z\star}_k$;
\item
$q$-eigenvalues (negative);
\end{itemize}
$y_j$-eigenvalues with $j$ not in $G$ are positive if $y_j$ is
$s^{x\star}_i$-compliant and negative else; $y_j$-eigenvalues with $j$ in $G$
but have a null biomass $Y_j(s^{x\star}_i)$ give negative eigenvalues. For the
remaining C-species, and species $x_i$, we obtain the following system: 
\[
\left\{\begin{array}{lll}
\dot x_i &=& \left(\alpha_i(s_{in}-x_i-\sum_ly_l) - D\right) x_i \\
	\dot y_j &=& \left(\beta_j(s_{in}-x_i-\sum_ly_l ,y_j) - D\right) y_j \\
\end{array}\right.
\]
and the Jacobian matrix:
\[
	\left(
	\begin{array}{cccccc}
-a_0 & -a_0 & \hdots & -a_0 & \hdots & -a_0 \\
-a_1 & -a_1-b_1 & \hdots & -a_1 & \hdots & -a_1\\
\vdots &\vdots & \ddots & \vdots &  & \vdots \\
-a_j & -a_j &  \hdots & -a_j-b_j & \hdots & -a_j \\
\vdots & \vdots &  & \vdots & \ddots & \vdots \\
-a_n & -a_n &  \hdots & -a_n & \hdots & -a_n-b_n \\
	\end{array}
	\right)
\]
with $a_0=\frac{\partial \alpha_i}{\partial s}x_i^\star>0$, $a_j =
\frac{\partial \beta_j}{\partial s}Y_j(s^{x\star}_i)>0$ and $b_j =
-\frac{\partial \beta_j}{\partial y_j}Y_j(s^{x\star}_i)>0$ (for
$j\,\in\,\{1,\cdots,n\}$).   
 This matrix has exactly the same form has the one considered on Appendix
 \ref{sub:C4}. The only difference being that the there is no ``$b_0$'' in the
 first element of the matrix. Defining a $b_0=0$, we can then conclude that
 all eigenvalues are real and negative because, following the development of
 Appendix \ref{sub:C4}, we obtain 
\[
A<-\min_j b_j=0
\]
 
Finally, only equilibrium $E^{(x,y)}_{1,\{1,\hdots,N_y\}}$ can be stable if
and only if $s^\star=s^{x\star}_1$. %We showed in this paper that in this case
                                %this equilibrium is globally asymptotically
                                %stable. 

\subsection{Q-coexistive equilibria}\label{sub:C6}
In this section we consider equilibria $E^{(z,y)}_{k,G}$ where phytoplankton
species $z_k$ coexists with the attached species in $G$, a subset of
$\{1,\hdots,N_y\}$, under substrate concentration $s^{z\star}_k$. 

We obtain here all the 
\begin{itemize}
\item
$x$-eigenvalues whose signs are the signs of $s^{z\star}_k-s^{x\star}_i$;
\item
$z_l$-eigenvalues ($l \neq j$) whose sign are the signs of $s^{z\star}_k-s^{z\star}_l$;
\item
$q$-eigenvalues (negative);
\end{itemize}
$y_j$-eigenvalues with $j$ not in $G$ are positive if $y_j$ is
$s^{z\star}_k$-compliant and negative else; $y_j$-eigenvalues with $j$ in $G$
but have a null biomass $Y_j(s^{z\star}_k)$ give negative eigenvalues. For the
remaining C-species, and species $z_k$, we obtain the following model 
\[
\left\{\begin{array}{lll}
	\dot y_j &=& (\beta_j(s_{in}-\sum_ly_l-q_kz_k,y_j) - D) y_j\\ 
	\dot z_k &=& (\gamma_k(q_k) - D) z_k \\
	\dot q_k &=& \rho_k(s_{in}-\sum_ly_l-q_kz_k) - f_k(q_k) 
\end{array}\right.
\]
and, swapping the last two equations and using $f_k(q_k)=\gamma_k(q_k)q_k$, we get the Jacobian matrix:
\[
	\left(
	\begin{array}{ccccccc}
-a_1-b_1 & \hdots & -a_1 & \hdots & -a_1 & -a_1 z_k & -a_1 q_k \\
\vdots & \ddots & \vdots &  & \vdots & \vdots & \vdots\\
-a_j &  \hdots & -a_j-b_j & \hdots & -a_j & -a_j z_k & -a_j q_k \\
\vdots &  & \vdots & \ddots & \vdots & \vdots \\
-a_n &  \hdots & -a_n & \hdots & -a_n-b_n & -a_n z_k & -a_n q_k \\
-a_{n+1} & \hdots & -a_{n+1}  & \hdots & -a_{n+1} & -a_{n+1} z_k-b_{n+1}q_k-\gamma & -a_{n+1} q_k\\
0 & \hdots & 0 & \hdots & 0 &  b_{n+1} z_k&0
	\end{array}
	\right)
\]
with $a_j = \frac{\partial \beta_j}{\partial s}Y_j(s^{z\star}_k)>0$ and $b_j =
-\frac{\partial \beta_j}{\partial y_j}Y_j(s^{z\star}_k)>0$ for
$j\,\in\,\{1,\cdots,n\}$, with $a_{n+1}=\frac{\partial \rho_k}{\partial s}$
and $b_{n+1}=\frac{\partial \gamma_k}{\partial q_k}$. 
%The $n+1$-th line and the last columnn give eigenvalue $-\frac{\partial \gamma_k}{\partial q_k} z_k^\star$ and can be removed and give matrix
%\[
%	\left(
%	\begin{array}{ccccccc}
%-a_1-b_1 & \hdots & -a_1 & \hdots & -a_1 & -a_1 q_k \\
%\vdots & \ddots & \vdots &  & \vdots & \vdots \\
%-a_j &  \hdots & -a_j-b_j & \hdots & -a_j & -a_j q_k \\
%\vdots &  & \vdots & \ddots & \vdots \\
%-a_n &  \hdots & -a_n & \hdots & -a_n-b_n & -a_n q_k \\
%-a_{n+1} & \hdots & -a_{n+1}  & \hdots & -a_{n+1}& -a_{n+1}q_k-b_{n+1}
%	\end{array}
%	\right)
%\]
By using the definition of eigenvalue $\lambda=(A + B i)$ (see
(\ref{eq:def_eigenvalue})) we follow a similar path to that of Appendix
\ref{sub:C4}, we show that the eigenvalues are real and negative.  %obtain on
                                %line $j$ 
%\[
%-b_j y_j - a_j \sum_{l=1}^n y_l - a_j q_ky_{n+1} = (A+Bi) y_j
%\]
%and on line $n+1$
%\[
%-a_{n+1}\sum_{l=1}^n y_l-a_{n+1}q_ky_{n+1}-b_{n+1}y_{n+1} = (A+Bi) y_{n+1}
%\]
%Isolating $y_j$ in the former and $q_ky_{n+1}$ in the latter then yields
%\[
%y_j=\frac{-a_j(\sum_{l=1}^n y_l+q_ky_{n+1})}{b_j+A+Bi}
%\]
%and
%\[
%q_ky_{n+1}=\frac{-a_{n+1}q_k(\sum_{l=1}^n y_l+q_ky_{n+1})}{b_{n+1}+A+Bi}
%\]
%Summing these equations over all the indices then gives
%\[
%\sum_{j=1}^n y_j+q_ky_{n+1}=-\sum_{j=1}^n \frac{a_j}{b_j+A+Bi}(\sum_{l=1}^n y_l+q_ky_{n+1})-\frac{-a_{n+1}q_k}{b_{n+1}+A+Bi}(\sum_{l=1}^n y_l+q_ky_{n+1})
%\]
%Proceeding as in Appendix \ref{sub:C4}, we show that the eigenvalues are real and such that
%\[
%A<-\min_{j\,\in\,\{1,\cdots,n+1\}}b_j=0
%\]

Finally, only equilibrium $E^{(z,y)}_{1,\{1,\hdots,N_y\}}$ can be stable if
and only if $s^\star=s^{z\star}_1$. %We showed in this paper that in this case
                                %this equilibrium is globally asymptotically
                                %stable. 

\begin{remark}
The same work can be done for the whole system (\ref{model_norm}), where the
eigenvalues are the same, plus the $-D$ eigenvalue which arises from mass
balance dynamics (\ref{eq:dot-M}). 
\end{remark}

%For acknowledgements section, please don't number the section, you need to begin with \section*{Acknowledgements}
% \section*{Acknowledgements} We would like to thank 

% \begin{thebibliography}{99}
\bibliographystyle{ieeetr}
\bibliography{biblio}

\begin{thebibliography}{10}

\bibitem{ArmMcG1980}
R.~Armstrong and R.~McGehee, ``Competitive exclusion,'' {\em American
  Naturalist}, vol.~115, p.~151, 1980.

\bibitem{SmiWal1995}
H.~Smith and P.~Waltman, {\em The theory of the chemostat. Dynamics of
  microbial competition}.
\newblock Cambridge Studies in Mathematical Biology. Cambridge University
  Press, 1995.

\bibitem{HsuHsu2008}
S.-B. Hsu and T.-H. Hsu, ``Competitive exclusion of microbial species for a
  single nutrient with internal storage,'' {\em SIAM J. Appl. Math.}, vol.~68,
  pp.~1600--1617, 2008.

\bibitem{TilSte1984}
D.~Tilman and R.~Sterner, ``Invasions of equilibria: tests of resource
  competition using two species of algae,'' {\em Oecologia}, vol.~61, no.~2,
  pp.~197--200, 1984.

\bibitem{HanHub1980}
S.~Hansen and S.~Hubell, ``Single-nutrient microbial competition: qualitative
  agreement between experimental and theoretically forecast outcomes,'' {\em
  Science}, vol.~207, no.~4438, pp.~1491--1493, 1980.

\bibitem{Til1977}
D.~Tilman, ``Resource competition between plankton algae: An experimental and
  theoretical approach.,'' {\em Ecology}, vol.~58, no.~22, pp.~338--348, 1977.

\bibitem{GroMazRap2005}
F.~Grognard, F.~Mazenc, and A.~Rapaport, ``Polytopic lyapunov functions for
  persistence analysis of competing species,'' {\em Discrete and Continuous
  Dynamical Systems-Series B}, vol.~8, no.~1, pp.~73--93, 2007.

\bibitem{Pulz2004}
O.~Pulz and W.~Gross, ``Valuable products from biotechnology of microalgae,''
  {\em Applied Microbiology and Biotechnology}, vol.~65, pp.~635--648, 2004.
\newblock 10.1007/s00253-004-1647-x.

\bibitem{Spolaore2006}
P.~Spolaore, C.~Joannis-Cassan, E.~Duran, and A.~Isambert, ``Commercial
  applications of microalgae,'' {\em Journal of Bioscience and Bioengineering},
  vol.~101, pp.~87--96, Feb. 2006.

\bibitem{Chisti2007}
Y.~Chisti, ``Biodisel from microalgae,'' {\em Biotechnology Advances}, vol.~25,
  pp.~294--306, 2007.

\bibitem{Hutchinson1961}
G.~E. Hutchinson, ``The paradox of the plankton,'' {\em The American
  Naturalist}, vol.~95, p.~137, 1961.

\bibitem{Dro1968}
M.~Droop, ``Vitamin $b_{12}$ and marine ecology,'' {\em J. Mar. Biol Assoc.
  U.K.}, vol.~48, pp.~689--733, 1968.

\bibitem{SciandraRamani}
A.~Sciandra and P.~Ramani, ``The steady states of continuous cultures with low
  rates of medium renewal per cell,'' {\em J. Exp. Mar. Biol. Ecol.}, vol.~178,
  pp.~1--15, 1994.

\bibitem{VBJM06}
I.~Vatcheva, H.~deJong, O.~Bernard, and N.~Mars, ``Experiment selection for the
  discrimination of semi-quantitative models of dynamical systems,'' {\em
  Artif. Intel.}, vol.~170, pp.~472--506, 2006.

\bibitem{Har1960}
G.~Hardin, ``The competitive exclusion principle,'' {\em Science}, vol.~131,
  no.~3409, pp.~1292--1297, 1960.

\bibitem{Elt1927}
C.~Elton, {\em Animal Ecology}.
\newblock Sidgwick \& Jackson, LTD. London, 1927.

\bibitem{Dar1859}
C.~Darwin, {\em On the Origin of Species by Means of Natural Selection, or the
  Preservation of Favoured Races in the Struggle for Life}.
\newblock John Murray, 1859.

\bibitem{Scr1959}
M.~Scriven, ``Explanation and prediction in evolutionnary theory,'' {\em
  Science}, vol.~130, no.~3374, pp.~477--482, 1959.

\bibitem{JesBoh2005}
C.~Jessup, S.~Forde, and B.~Bohannan, ``Microbial experimental systems in
  ecology,'' {\em Advances in Ecological Research}, vol.~37, pp.~273--306,
  2005.

\bibitem{Gau1934}
G.~Gause, {\em The Struggle for Existence}.
\newblock Williams and Wilkins, Baltimore, 1934.

\bibitem{Mon1942}
J.~Monod, ``Reserches sur la croissance des cultures bacteriennes,'' {\em
  Paris: Herrmann et Cie}, 1942.

\bibitem{CapMey72}
J.~Caperon and J.~Meyer, ``Nitrogen-limited growth of marine phytoplankton. i.
  changes in population characteristics with steady-state growth rate,'' {\em
  Deep-Sea Res.}, vol.~19, pp.~601--618, 1972.

\bibitem{LangOyar92}
K.~Lange and F.~J. Oyarzun, ``The attractiveness of the {D}roop equations,''
  {\em Mathematical Biosciences}, vol.~111, pp.~261--278, 1992.

\bibitem{OyaLan94}
F.~J. Oyarzun and K.~Lange, ``The attractiveness of the {D}roop equations.
  {II}: Generic uptake and growth functions,'' {\em Mathematical Biosciences},
  vol.~121, pp.~127--139, 1994.

\bibitem{BerGou95}
O.~Bernard and J.-L. Gouz\'e, ``Transient behavior of biological loop models,
  with application to the {D}roop model,'' {\em Mathematical Biosciences},
  vol.~127, no.~1, pp.~19--43, 1995.

\bibitem{Contois1959}
D.~Contois, ``Kinetics of bacterial growth: relationship between population
  density and species growth rate of continuous cultures,'' {\em J Gen
  Microbiol.}, pp.~40--50, 1959.

\bibitem{LeoTum75}
J.~Leon and D.~Tumpson, ``Competition between two species of two complementary
  or substitutable resources,'' {\em J. Theor. Biol.}, vol.~50, pp.~185--201,
  1975.

\bibitem{HsuCheHub81}
S.~Hsu, K.~Cheng, and S.~Hubbel, ``Exploitative competition of micro-organisms
  for two complementary nutrients in continuous culture,'' {\em SIAM J. Appl.
  Math.}, vol.~41, pp.~422--444, 1981.

\bibitem{FSW89}
H.~Freedman, J.~So, and P.~Waltman, ``Coexistence in a model of competition in
  the chemostat incorporating discrete delays,'' {\em SIAM J. Appl. Math.},
  vol.~49, pp.~859--870, 1989.

\bibitem{DelSmi03}
P.~de~Leenheer and H.~Smith, ``Feedback control for the chemostat,'' {\em J.
  Math. Biol.}, vol.~46, pp.~48--70, 2003.

\bibitem{DelAngSon03}
P.~de~Leenheer, D.~Angeli, and A.~Sontag, ``A feedback perspective for
  chemostat models with crowding effects,'' in {\em Positive Systems}, vol.~294
  of {\em Lecture Notes in Control and Inform. Sci}, pp.~167--174,
  Springer-Verlag, 2003.

\bibitem{APW03}
J.~Arino, S.~Pilyugin, and G.~Wolkowicz, ``Considerations on yield, nutrient
  uptake, cellular growth, and competition in chemostat models,'' {\em Canadian
  Applied Math Quarterly}, vol.~11, pp.~107--142, (2003) [2005].

\bibitem{Wil90}
J.~B. Wilson, ``Mechanisms of species coexistence: twelve explanations for the
  hutchinson's 'paradox of the phytoplankton': evidence from new zealand plant
  communities,'' {\em New Zealand journal of Ecology}, vol.~137, pp.~17--42,
  1990.

\bibitem{FreSte1981}
A.~Fredrickson and G.~Stephanopoulos, ``Microbial competition,'' {\em Science},
  vol.~213, pp.~972--979, 1981.

\bibitem{GouRob2005}
J.~Gouz\'e and G.~Robledo, ``Feedback control for nonmonotone competition
  models in the chemostat,'' {\em Nonlinear Analysis: Real World Applications},
  pp.~671--690, 2005.

\bibitem{LeeLiSmi2003}
P.~de~Leenheer, B.~Li, and H.~Smith, ``Competition in the chemostat : some
  remarks,'' {\em Canadian applied mathematics quarterly}, vol.~11, no.~2,
  pp.~229--247, 2003.

\bibitem{RaoRox1990}
N.~Rao and E.~Roxin, ``Controled growth of competing species,'' {\em Journal on
  Applied Mathematics}, vol.~50, no.~3, pp.~853--864, 1990.

\bibitem{MasGroBer08}
P.~Masci, O.~Bernard, and F.~Grognard, ``Continuous selection of the fastest
  growing species in the chemostat,'' in {\em Proceedings of the IFAC
  conference}, Seoul, Korea, 2008.

\bibitem{Thieme1992}
H.~R. Thieme, ``Convergence results and a {P}oicar{\'e}-{B}endixson trichotomy
  for asymptotically autonomous differential equations,'' {\em Journal of
  {M}athematical {B}iology}, vol.~30, pp.~755--763, Aug. 1992.

\bibitem{Diekmann2003}
O.~Diekmann, ``A beginner's guide to adaptive dynamics,'' {\em Banach Center
  Publ.}, vol.~63, pp.~47--86, 2003.

\end{thebibliography}

\medskip
%% The data information below will be filled by AIMS editorial staff
Received September 2006; revised February 2007.

\medskip

\end{document}